\documentclass[11pt]{article}

\usepackage[margin=1in]{geometry}
\usepackage{amsmath, amsthm, amsfonts, amssymb, mathrsfs}
\usepackage{enumitem}
\usepackage{hyperref}
\usepackage{xcolor}
   \usepackage{tikz}
   \usepackage{caption}
   \usetikzlibrary{calc}

\newtheorem{main}{Theorem}

\newtheorem{thm}{Theorem}[section]
\newtheorem{prop}[thm]{Proposition}
\newtheorem{lemma}[thm]{Lemma}

\theoremstyle{definition}
\newtheorem{defn}[thm]{Definition}
\newtheorem{rmk}[thm]{Remark}

\newcommand{\R}{\mathbb{R}}
\newcommand{\Rpos}{\R^+}

\title{Regularity of Lyapunov exponents at one-point Lyapunov spectra: the semisimple case}
\author{Yingjian Liu, Marcelo Viana}
\date{\today}

\begin{document}

\maketitle

\begin{abstract}
    We study the regularity of Lyapunov exponents as functions on the space of compactly supported probability measures on $\mathrm{GL}(d,\mathbb{R})$. We prove that the Lyapunov exponents are pointwise log-Hölder continuous with respect to the Wasserstein distance, at semisimple probability measures with
    one-point Lyapunov spectrum. The proof relies on a decomposition of the
    action into virtually conformal subspaces and a Berry-Esseen type estimate
    for the random walk towards these subspaces.
\end{abstract}

\tableofcontents

\section{Introduction and main result}
Let $G=\mathrm{GL}(d,\mathbb{R})$ be the group of $d\times d$ invertible real matrices equipped with the distance derived from the operator norm 
$\|g\| = \sup\left\{\|gv\|/\|v\|: v \in \mathbb{R}^{d} \setminus\{0\}\right\}$.
Let $\sigma_1(g) \geq \sigma_2(g) \geq \dots \geq \sigma_d(g)>0$ denote the
singular values of any matrix $g \in G$. For any probability measure $\mu$ on
$G$ satisfying the finite first moment condition
$$
\int _{G} \log   N(g)  \, d\mu(g) < \infty,
$$
with $N(g)=\max\left( \|g\|, \|g^{-1}\| \right)$,
the \textit{Lyapunov exponents} $\lambda_{1}(\mu) \geq\dots\geq \lambda _{d}(\mu)$ may
be defined by
$$
\lambda _{i}(\mu) = \lim_{ n \to \infty } \frac{1}{n} \log\sigma_i(g_{n} \cdots g_{1}),
$$
where $(g_n)_{n \in \mathbb{N}}$ are independent random variables distributed according to $\mu$ and the equality holds almost surely.
See Furstenberg, Kesten~\cite{FK60}, Oseledets~\cite{Ose68}.

Let $\mathscr{C}(G)$ be the Banach space of bounded continuous real functions on $G$ and $\mathscr{P}(G)$ be the space of probability measures on $G$ endowed
with the \textit{Wasserstein distance}
$$
d_{\mathrm W}(\mu', \mu) := \sup \left\{ \left| \int_{G} \varphi \, d(\mu'-\mu) \right|: \varphi \in \mathscr{C}(G), \ \mathrm{Lip}(\varphi) \leq 1  \right\},
$$
where $\mathrm{Lip}(\cdot)$ is the Lipschitz constant.

Denote by $\Gamma _{\mu}$ the subsemigroup of $\mathrm{GL}(d, \R)$ generated by the support $\mathrm{supp}(\mu)$ of any $\mu\in\mathscr{P}(G)$.
We call $\mu$ \textit{semisimple} if the canonical action of $\Gamma _{\mu}$ on $\mathbb{R}^{d}$ is semisimple, meaning that any $\Gamma_\mu$-invariant subspace
admits a $\Gamma_\mu$-invariant complementary subspace.

Our main result in this paper is:

\begin{main} \label{thmA}
Let $\mu\in\mathscr{P}(G)$ be compactly supported, semisimple,
and such that $\lambda_{1}(\mu)=\lambda _{d}(\mu)$.
Then the Lyapunov exponent functions $\eta \mapsto \lambda_i(\eta)$ are log-H\"older continuous at $\mu$ in the following sense:
there exists a universal constant $\gamma>0$ such that, given any compact set $K \subset G$ containing $\mathrm{supp}(\mu)$, there
exist constants $\epsilon=\epsilon(K)>0$ and
$C=C(K,\mu)>0$ such that for any $\mu' \in \mathscr{P}(G)$ with
$\mathrm{supp}(\mu') \subset K$ and $d_{\mathrm W}(\mu',\mu)<\epsilon$,
\begin{equation} \label{eq:main_bound}
\left| \lambda_{i}(\mu') - \lambda_{i}(\mu) \right|
\leq C \left| \log d_{\mathrm W}(\mu', \mu) \right|^{-\gamma} 
\text{ for every } i\in \{1,\ldots,d\}.
\end{equation}
\end{main}

In particular, the Lyapunov exponent functions $\lambda_i(\cdot)$ are log-H\"older continuous at $\mu$ with respect to the Wasserstein-Hausdorff distance 
(see Tal, Viana~\cite{TVi20}) in the space of compactly supported probability measures on $G$.

\subsection{Context}

The study of the Lyapunov exponents of random matrices as functions of the underlying probability distribution has attracted a lot of attention since 
the foundational work of Furstenberg~\cite{Fur63}.
In the following we put Theorem~\ref{thmA} in context within that theory.

Ruelle~\cite{Rue79a} proved real analyticity of the top Lyapunov exponent
$\lambda_1(\cdot)$ when there is an invariant convex cone.
Real analyticity on probability weights still holds under the much weaker assumption that the
Lyapunov exponent is simple: that was first proved by Peres~\cite{Pe91} when
the probability distribution is finitely supported, and it has just been
extended to great generality by Amorim, Dur\~{a}es, Melo~\cite{ADM,Amorim}.

Furstenberg, Kifer~\cite{FK83} and Hennion~\cite{Hen84} proved continuity of
the top Lyapunov exponent at every probability distribution for which there
exists at most one invariant subspace.
In the compactly supported case, which interests us more directly here,
continuity actually holds everywhere: that was proved by Bocker, Viana~\cite{BoV17} when $d=2$ and by Avila, Eskin, Viana~\cite{AEV} in all dimensions. See also Viana~\cite[Chapter~10]{LLE}.
The conclusions of \cite{BoV17} have been extended to the Markov case,
by Malheiro, Viana~\cite{MaV15}, and to a very broad setting of linear
cocycles with invariant holonomies, by Backes, Brown, Butler~\cite{BBB18}.

A classical result of Le Page~\cite{LP89}, with extention by \cite{DuK16b}, \cite{BaD19} and \cite{Du25}, asserts that, in any dimension, the top Lyapunov exponent is H\"{o}lder continuous on every compact set of probability distributions satisfying strong irreducibility and the contraction property.
A construction of Duarte, Klein, Santos~\cite{DKS19} shows that these assumptions
can not be removed. See also Tal, Viana~\cite[Section~5]{TVi20}.

More recently, Tal, Viana~\cite{TVi20} proved that pointwise H\"older continuity still holds under the weaker assumption that the Lyapunov spectrum is simple, at least in dimension $2$. They also find explicit modulus of continuity -- log-H\"older continuity -- that holds at every point. A partial extension to arbitrary
dimensions of the pointwise H\"{o}lder continuity result has just been given by
Barreto~\cite{Barreto}.

Independently, and at about the same time, Duarte, Klein~\cite{DuK19b} proved that, for 2-dimensional matrices, the Lyapunov exponents are weak-H\"{o}lder continuous functions of the probability weights and matrix coefficients on any compact domain where the top exponent is simple. \textit{Weak-H\"{o}lder} is defined by replacing $d(x,y)^\beta$ with $\exp(-\beta(-\log d(x,y))^{\theta})$ with $\theta<1$ in the definition of H\"older continuity.

We refer the reader to the survey paper Viana~\cite{Disc20} and the books Viana~\cite{LLE} and~\cite{DuK16b} for much more information on the background to our present results, and additional references.

\subsection{Outline of the proof}

To illustrate the core mechanism, consider the measure
$\mu = \frac{1}{3}\delta_A + \frac{1}{3}\delta_B + \frac{1}{3}\delta_C$,
whose support consists of the three commuting positive diagonal matrices
\begin{equation}\label{eq:matrices}
  A = \begin{pmatrix} 2 & 0 & 0 \\ 0 & \tfrac{1}{2} & 0 \\ 0 & 0 & 1 \end{pmatrix}, \qquad
  B = \begin{pmatrix} 1 & 0 & 0 \\ 0 & 2 & 0 \\ 0 & 0 & \tfrac{1}{2} \end{pmatrix}, \qquad
  C = \begin{pmatrix} \tfrac{1}{2} & 0 & 0 \\ 0 & 1 & 0 \\ 0 & 0 & 2 \end{pmatrix}.
\end{equation}
Since the matrices commute, the Lyapunov exponents reduce to the expected values of
the logarithms of the diagonal entries. The cyclic structure of the support ensures
complete cancellation,
\[
  \lambda_i = \tfrac{1}{3}\log 2 + \tfrac{1}{3}\log \tfrac{1}{2} + \tfrac{1}{3}\log 1 = 0,
\]
yielding a one-point spectrum. Moreover, the system is semisimple: the state
space decomposes into one-dimensional invariant subspaces $V^i = \mathbb{R}e_i$
for $i = 1, 2, 3$.

The proof proceeds by analysing the induced random walk on the standard simplex
$\Delta^2$ (Figure~\ref{fig:simplex}). The simplex is partitioned into 
neighbourhoods $B^1, B^2, B^3$ around the axes
$\mathbb{R}e_i$, and a complementary central region $B_r^c$. The key estimate
controls transition probabilities between these regions: from any point in $B_r^c$,
the probability of remaining in $B_r^c$ at time $n$ decays polynomially
in $n$. These bounds establish Theorem~\ref{thmA}.

\begin{figure}[ht]
  \centering
  \begin{tikzpicture}[
      scale=0.5,
      x={(-0.866cm,-0.5cm)},
      y={(0.866cm,-0.5cm)},
      z={(0cm,1cm)},
      line join=round,
      line cap=round
  ]
 
  \pgfmathsetmacro{\rad}{4}
  \pgfmathsetmacro{\radplus}{\rad+1.5}
  \pgfmathsetmacro{\ang}{25}
  \pgfmathsetmacro{\compang}{90-\ang}
  \pgfmathsetmacro{\radcos}{\rad*cos(\ang)}
  \pgfmathsetmacro{\radsin}{\rad*sin(\ang)}
  \pgfmathsetmacro{\nL}{\rad*0.9}
  \pgfmathsetmacro{\nS}{\rad*0.15}
 
  % Axes
  \draw[very thick, ->] (0,0,0) -- ({\radplus},0,0)
    node[anchor=north east, font=\large] {$e_1$};
  \draw[very thick, ->] (0,0,0) -- (0,{\radplus},0)
    node[anchor=north west, font=\large] {$e_2$};
  \draw[very thick, ->] (0,0,0) -- (0,0,{\radplus})
    node[anchor=south, font=\large] {$e_3$};
 
  % Central region B_r^c
  \fill[gray!20, opacity=0.85]
    plot[variable=\t, domain=0:90, samples=30, smooth]
      ({\rad*cos(\t)}, {\rad*sin(\t)}, 0) --
    plot[variable=\t, domain=0:90, samples=30, smooth]
      (0, {\rad*cos(\t)}, {\rad*sin(\t)}) --
    plot[variable=\t, domain=90:0, samples=30, smooth]
      ({\rad*sin(\t)}, 0, {\rad*cos(\t)}) -- cycle;
 
  % B^1
  \fill[blue!30, opacity=0.85]
    plot[variable=\t, domain=0:\ang, samples=20, smooth]
      ({\rad*cos(\t)}, {\rad*sin(\t)}, 0) --
    plot[variable=\t, domain=0:90, samples=20, smooth]
      ({\radcos}, {\radsin*cos(\t)}, {\radsin*sin(\t)}) --
    plot[variable=\t, domain=\ang:0, samples=20, smooth]
      ({\rad*cos(\t)}, 0, {\rad*sin(\t)}) -- cycle;
 
  % B^2
  \fill[blue!30, opacity=0.85]
    plot[variable=\t, domain=90:\compang, samples=20, smooth]
      ({\rad*cos(\t)}, {\rad*sin(\t)}, 0) --
    plot[variable=\t, domain=0:90, samples=20, smooth]
      ({\radsin*cos(\t)}, {\radcos}, {\radsin*sin(\t)}) --
    plot[variable=\t, domain=\ang:0, samples=20, smooth]
      (0, {\rad*cos(\t)}, {\rad*sin(\t)}) -- cycle;
 
  % B^3
  \fill[blue!30, opacity=0.85]
    plot[variable=\t, domain=90:\compang, samples=20, smooth]
      (0, {\rad*cos(\t)}, {\rad*sin(\t)}) --
    plot[variable=\t, domain=90:0, samples=20, smooth]
      ({\radsin*cos(\t)}, {\radsin*sin(\t)}, {\radcos}) --
    plot[variable=\t, domain=\compang:90, samples=20, smooth]
      ({\rad*cos(\t)}, 0, {\rad*sin(\t)}) -- cycle;
 
  % Outer arc boundaries
  \draw[thick]
    plot[variable=\t, domain=0:90, samples=30, smooth]
      ({\rad*cos(\t)}, {\rad*sin(\t)}, 0);
  \draw[thick]
    plot[variable=\t, domain=0:90, samples=30, smooth]
      (0, {\rad*cos(\t)}, {\rad*sin(\t)});
  \draw[thick]
    plot[variable=\t, domain=0:90, samples=30, smooth]
      ({\rad*cos(\t)}, 0, {\rad*sin(\t)});
 
  % Dashed neighbourhood cuts
  \draw[thick, dashed]
    plot[variable=\t, domain=0:90, samples=20, smooth]
      ({\radcos}, {\radsin*cos(\t)}, {\radsin*sin(\t)});
  \draw[thick, dashed]
    plot[variable=\t, domain=0:90, samples=20, smooth]
      ({\radsin*cos(\t)}, {\radcos}, {\radsin*sin(\t)});
  \draw[thick, dashed]
    plot[variable=\t, domain=0:90, samples=20, smooth]
      ({\radsin*cos(\t)}, {\radsin*sin(\t)}, {\radcos});
 
  % Region labels
  \node[font=\large] at (\nL, \nS, \nS) {$B^1$};
  \node[font=\large] at (\nS, \nL, \nS) {$B^2$};
  \node[font=\large] at (\nS, \nS, \nL) {$B^3$};
  \node[font=\large] at ({\rad/3}, {\rad/3}, {\rad/3}) {$B_r^c$};
 
  \end{tikzpicture}
  \caption{The standard simplex $\Delta^2$ partitioned into the 
    neighbourhoods $B^1, B^2, B^3$ around each axis,
    and the central region $B_r^c$.}
  \label{fig:simplex}
\end{figure}
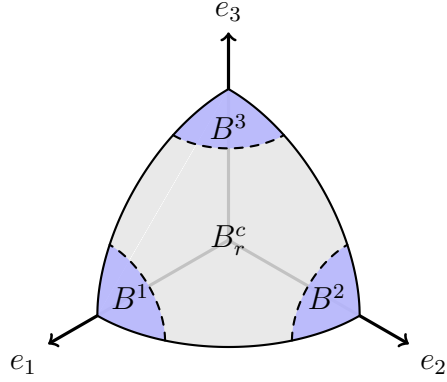

The paper proceeds in four stages, reducing the analysis
from algebraic structure down to probabilistic convergence bounds.
 
\begin{enumerate}
  \item {Algebraic structure} (Section~\ref{sec:alg}). When $\mu$ is not
    conformal, the one-point spectrum hypothesis combined with Furstenberg's
    criterion yields a coarsest decomposition of $\mathbb{R}^d$ into virtually
    $\Gamma_\mu$-conformal subspaces. To handle periodic behaviour of the random
    walk across these subspaces, we pass to a suitable aperiodic convolution power
    $\mu^{*p_\mu}$.
 
  \item {Analytical reduction} (Section~\ref{sec:reduction}). Via
    Furstenberg's formula and the Lipschitz continuity of the Markov operator, the
    difference $|\lambda_1(\mu') - \lambda_1(\mu)|$ is translated into a convergence
    rate estimate for the $\mu'$-stationary measure under the unperturbed operator
    $P_\mu$. The integral is split into a neighbourhood of the virtually invariant
    conformal subspaces, where the measure accumulates, and its complement.
 
  \item {Random walk estimates} (Section~\ref{sec:rw}). The logarithmic
    slopes between conformal blocks are analysed as one-dimensional random walks.
    Adapting Schneider's theorem for exponentially $\varphi$-mixing sequences yields
    polynomial accumulation and equidistribution bounds for the Markov iterates.
 
  \item {Optimisation} (Section~\ref{sec:proof}). The polynomial convergence
    to the stationary measures is balanced against the exponential growth of the
    Lipschitz constants of the iterated operators. Optimising the iteration step $n$
    in terms of $d_{\mathrm{W}}(\mu', \mu)$ yields log-H\"{o}lder regularity for
    $\lambda_1$; the bound is then extended to the full Lyapunov spectrum via
    exterior powers.
\end{enumerate}
 
The necessary background on Markov operators, exterior powers, Furstenberg's
formula, and the conformal case is collected in Section~\ref{sec:prelim}.

% ----------------------------------------------------------------------
\section{Preliminaries}\label{sec:prelim}

We collect the analytic tools used throughout the paper: Lipschitz estimates for the Markov operator, exterior powers and convolutions, Furstenberg's formula, and a self-contained proof of pointwise Lipschitz continuity in the conformal case.

Let $X=\mathbb{P} (\mathbb{R}^{d})$ be the real projective space endowed with the distance
$$
d_{X}([u],[v]) = |\sin\angle(u,v)| = \frac{ \| u \wedge v \| }{ \| u \| \| v \| },
$$
where $[w]=\mathbb{R}w$ denotes the projective class of any non-zero $w\in\mathbb{R}^{d}$. Let $\mathscr{C}(X)$ be the Banach space of continuous functions on $X$ and $\mathscr{P}(X)$ be the space of probability measures on $X$ endowed with the \textit{Wasserstein distance} (see, e.g., Villani \cite{Vil09}):
$$
d_{\mathrm W}(\eta', \eta) := \sup \left\{ \left| \int_{X} \psi \, d(\eta'-\eta) \right|: \psi \in \mathscr{C}(X), \ \mathrm{Lip}(\psi) \leq 1  \right\} \leq 1.
$$

\subsection{Markov and Lipschitz operators}
We define the Lipschitz constant for elements in $\mathscr{C}(X)$ by
$$
\mathrm{Lip}(\varphi) = \sup_{x\neq y} \frac{| \varphi(x) - \varphi(y) | }{d_X(x,y)}
$$
and the Markov operator $P_\mu$ on these continuous functions by 
$$
P_\mu\varphi(x) = \int_G \varphi(gx) \, d\mu(g).
$$

\begin{lemma}\label{lemma:lip}
Let $K \subset G$ be compact and let $\mu, \mu'$ be probability measures on $G$ supported in $K$. Set $L_K = \sup_{g\in K} N(g)^2$.
\begin{enumerate}[label=(\roman*)]
    \item The Markov operator $P_{\mu}$ acts on Lipschitz functions and $\mathrm{Lip}(P_{\mu}\varphi) \leq L_K\,\mathrm{Lip}(\varphi)$.
    \item The dual of the Markov operator $P_\mu$, still denoted by $P_{\mu}$, is a Lipschitz map of $\mu$ in the Wasserstein metric. More precisely, for any $\varphi \in \mathscr{C}(X)$, 
$$
\|P_{\mu'}\varphi - P_{\mu}\varphi\|_{\infty} \leq \sqrt{L_K}\, \mathrm{Lip}(\varphi) \, d_{\mathrm W}(\mu',\mu).
$$
\end{enumerate}
\end{lemma}

\begin{proof}
(i) For any $x, y \in X$, a direct computation of the operator yields:
$$
\left| P_{\mu}\varphi(x) - P_{\mu}\varphi(y) \right| \leq \int_G \left| \varphi(gx) - \varphi(gy) \right| \, d \mu(g).
$$
Using the Lipschitz property of $\varphi$, we can bound the integrand:
$$
\int_G \left| \varphi(gx) - \varphi(gy) \right| \, d \mu(g) \leq \left( \int_G \frac{d_X(gx, gy)}{d_X(x,y)} \, d\mu(g) \right) \mathrm{Lip}(\varphi) \, d_X(x,y).
$$
Because the quotient is
$$\frac{d_X(gx, gy)}{d_X(x,y)}  = \frac{\| gu \wedge gv \| }{  \|u\wedge v\|}  \frac{\|u \|}{\|gu\|} \frac{\|v \|}{\|gv\|} 
\le \frac{ | \det(g|_{\mathbb{R}u \oplus \mathbb{R}v}) | }{ \sigma_2(g|_{\mathbb{R}u \oplus \mathbb{R}v}) ^2 } = \frac{\sigma_1(g|_{\mathbb{R}u \oplus \mathbb{R}v}) }{\sigma_2(g|_{\mathbb{R}u \oplus \mathbb{R}v}) } \le \frac{\sigma_1(g)}{\sigma_d(g)},$$ 
the integral ratio is bounded by $\sup_{g \in K} \|g\| \|g^{-1}\| \leq L_K$.

(ii) For any fixed $x = \mathbb{R}v \in X$ with $\|v\| = 1$, consider the function $f_x: G \to \mathbb{R}$ defined by $f_x(g) = \varphi(gx)$. We determine the Lipschitz constant of $f_x$ restricted to the compact set $K$. For any $g, g' \in K$, we evaluate the projective distance between their actions on $x$:
$$
d_X(g x, g' x) = |\sin \angle(g v, g' v)| \leq \frac{\|g v - g' v\|}{\|g' v\|}.
$$
The numerator is bounded by $\|g v - g' v\| \leq \|g - g'\| \|v\| = \|g - g'\|$. For the denominator, we apply the relation $\|v\| = \|(g')^{-1} g' v\| \leq \|(g')^{-1}\| \|g' v\|$, implying $\|g' v\| \geq \|(g')^{-1}\|^{-1}$. Substituting these bounds yields:
$$
d_X(g x, g' x) \leq \|(g')^{-1}\| \|g - g'\| \leq \left( \sup_{h \in K} \|h^{-1}\| \right) \|g - g'\|.
$$
Therefore, the composition $f_x$ satisfies $\mathrm{Lip}(f_x|_K) \leq \mathrm{Lip}(\varphi) \sup_{h \in K} \|h^{-1}\|$. Using the dual formulation of the Wasserstein distance evaluated at $x$:
$$
\left| P_{\mu'}\varphi(x) - P_{\mu}\varphi(x) \right| = \left| \int_G f_x(g) \, d(\mu' - \mu)(g) \right| \leq \mathrm{Lip}(f_x|_K) d_{\mathrm W}(\mu', \mu).
$$
Since $N(g) = \max(\|g\|, \|g^{-1}\|) \geq 1$, we have $\sup_{h \in K} \|h^{-1}\| \leq \sup_{h \in K} N(h) = \sqrt{L_K}$. Taking the supremum over all $x \in X$ yields the desired uniform bound.
\end{proof}

\subsection{Exterior powers and convolutions}
Two further constructions will be used repeatedly. The exterior powers allow us to reduce regularity of the full Lyapunov spectrum to that of the top exponent, and the convolution powers handle periodicity of the random walk on finite groups. We recall the relevant facts, following Viana~\cite{LLE}.

For each $k \in \{1, \dots, d\}$, let $\wedge^k \mathbb{R}^d$ denote the $k$-th exterior power of $\mathbb{R}^d$. Any matrix $g \in G$ induces an invertible linear map $\wedge^k g$ on $\wedge^k \mathbb{R}^d$ defined on decomposable vectors by 
$$
(\wedge^k g)(v_1 \wedge \dots \wedge v_k) = gv_1 \wedge \dots \wedge gv_k.
$$
We denote by $\wedge^k \mu$ the pushforward of $\mu$ under the map $g \mapsto \wedge^k g$. The operator norm on $\wedge^k \mathbb{R}^d$ induced by the standard Euclidean norm satisfies $\|\wedge^k g\| \le \|g\|^k$. A fundamental property of Lyapunov exponents is that the top Lyapunov exponent of the exterior power corresponds to the sum of the top $k$ exponents of the original measure:
$$
\lambda_1(\wedge^k \mu) = \sum_{l=1}^k \lambda_l(\mu).
$$
Consequently, we can recover the individual exponents via the difference $\lambda_k(\mu) = \lambda_1(\wedge^k \mu) - \lambda_1(\wedge^{k-1} \mu)$ for $k \ge 2$, with the convention $\lambda_1(\wedge^0 \mu) = 0$. This reduces the regularity analysis of the entire Lyapunov spectrum to the study of the first Lyapunov exponent on the exterior powers.

Additionally, to circumvent periodic behaviors in the random walk on finite groups, we define the $p$-th convolution $\mu^{*p}$ as the image of the product measure $\mu^{\otimes p}$ under the multiplication map:
$$
(g_1, \ldots, g_p) \mapsto g_p \cdots g_1.
$$

The following lemma records the multiplicativity of Lyapunov exponents under convolution powers, which will be used in the proof of Theorem~\ref{thmA} to reduce the periodic case to the aperiodic one.

\begin{lemma} \label{lemma:convolution_le}
Let $\mu$ be a probability measure on $G$ with a finite first moment. For any integer $p \in \mathbb{N}^*$, the Lyapunov exponents of the convolution power $\mu^{*p}$ satisfy
$$
\lambda_k(\mu^{*p}) = p \lambda_k(\mu)
$$
for every $k \in \{1, \ldots, d\}$.
\end{lemma}

\begin{proof}
Let $(g_n)_{n \in \mathbb{N}^*}$ be a sequence of independent random variables distributed according to $\mu$. Consider the block products $h_m = g_{mp} \cdots g_{(m-1)p+1}$ for $m \in \mathbb{N}^*$. The sequence $(h_m)_{m \in \mathbb{N}^*}$ consists of independent random variables distributed according to $\mu^{*p}$. By the definition of Lyapunov exponents,
$$
\lambda_k(\mu^{*p}) = \lim_{m \to \infty} \frac{1}{m} \log \sigma_k(h_m \cdots h_1) \quad \text{a.s.}
$$
Since $h_m \cdots h_1 = g_{mp} \cdots g_1$, we may rewrite this limit as
$$
\lambda_k(\mu^{*p}) = \lim_{m \to \infty} \frac{p}{mp} \log \sigma_k(g_{mp} \cdots g_1) = p \left( \lim_{n \to \infty} \frac{1}{n} \log \sigma_k(g_n \cdots g_1) \right) = p \lambda_k(\mu),
$$
where the last limit holds almost surely by the definition of Lyapunov exponents for $\mu$.
\end{proof}

The following lemma controls the Wasserstein distance under these operations.

\begin{lemma} \label{lemma:trivial}
Let $K \subset G$ be a compact set and let $\mu, \mu'$ be probability measures supported in $K$. For the constant $L_K \geq 1$ defined in Lemma \ref{lemma:lip}, we have for any $p \in \mathbb{N}^*$, 
$$
d_{\mathrm W}( (\mu')^{*p}, \mu ^{*p} ) \leq p \, L_K^{(p-1)/2} d_{\mathrm W}(\mu', \mu),
$$
and for any $k \in \{ 1, \ldots, d \}$,
$$
d_{\mathrm W}(\wedge ^{k}(\mu'), \wedge ^{k}\mu ) \leq k\, L_K^{(k-1)/2} d_{\mathrm W}(\mu', \mu) .
$$
\end{lemma}

\begin{proof}
For any $g \in K$, its operator norm is strictly bounded by $\|g\| \leq N(g) \leq \sqrt{L_K}$.

For the first inequality, let $\varphi \in \mathscr{C}(G)$ be a $1$-Lipschitz function. We rewrite the difference in the integrals with respect to the $p$-th convolutions as a telescopic sum where only one coordinate changes at a time:
$$
\int_G \varphi \, d((\mu')^{*p}) - \int_G \varphi \, d(\mu^{*p}) = \sum_{m=1}^p \int_{G^p} \varphi(g_p \dots g_{m+1} g g_{m-1} \dots g_1) \, d\mu'(g_p) \dots d(\mu' - \mu)(g) \dots d\mu(g_1).
$$
For a fixed choice of $g_i \in K$ (for $i \neq m$), consider the single-variable function 
$$f(g) = \varphi(g_p \dots g_{m+1} g\, g_{m-1} \dots g_1).$$ 
The distance between two such products evaluated at $g$ and $g'$ is bounded using the submultiplicativity of the operator norm:
$$
\| g_p \dots g_{m+1} g\, g_{m-1} \dots g_1 - g_p \dots g_{m+1} g'\, g_{m-1} \dots g_1 \| \leq \left( \prod_{i > m} \|g_i\| \right) \|g - g'\| \left( \prod_{i < m} \|g_i\| \right).
$$
Because $\varphi$ is $1$-Lipschitz, the Lipschitz constant of $f$ on $K$ satisfies:
$$
\mathrm{Lip}(f|_K) \leq \prod_{i \neq m} \|g_i\| \leq \prod_{i \neq m} \sqrt{L_K} = L_K^{(p-1)/2}.
$$
Therefore, the integral with respect to $d(\mu' - \mu)(g)$ in each of the $p$ terms is bounded by $L_K^{(p-1)/2} d_{\mathrm W}(\mu', \mu)$. Summing over all $p$ components yields the first bound.

For the second inequality, let $\psi \in \mathscr{C}(\mathrm{GL}(\wedge^k \mathbb{R}^d))$ be a $1$-Lipschitz function. We aim to bound the difference:
$$
\left| \int_G \psi(\wedge^k g) \, d(\mu' - \mu)(g) \right|.
$$
This requires finding the Lipschitz constant of the map $g \mapsto \wedge^k g$ restricted to $K$. For any $g, g' \in K$, we express the difference of their exterior powers using another telescopic sum:
$$
\wedge^k g - \wedge^k g' = \sum_{m=1}^k (\wedge^{k-m} g) \wedge (g - g') \wedge (\wedge^{m-1} g').
$$
Taking the operator norm on the exterior algebra and using the geometric fact that $\|\wedge^j g\| \leq \|g\|^j$, we obtain:
$$
\|\wedge^k g - \wedge^k g'\| \leq \sum_{m=1}^k \|g\|^{k-m} \|g - g'\| \|g'\|^{m-1} \leq \sum_{m=1}^k (\sqrt{L_K})^{k-m} (\sqrt{L_K})^{m-1} \|g - g'\| = k L_K^{(k-1)/2} \|g - g'\|.
$$
Thus, the map $g \mapsto \wedge^k g$ is $(k L_K^{(k-1)/2})$-Lipschitz on $K$, and the bound on the Wasserstein distance follows immediately.
\end{proof}

\subsection{Furstenberg's formula and a Lipschitz bound}
Denote the norm cocycle by 
$$
\Phi(g, \mathbb{R}v) = \log \frac{\|gv\|}{\|v\|},
$$ 
and the drift function by 
$$
\phi _{\mu}(x) = \int_G \Phi(g,x) \, d\mu(g).
$$

\begin{lemma} \label{lemma:drift_lip}
Let $K \subset G$ be a compact set and let $\mu$ be a probability measure supported in $K$. For $L_K$ defined in Lemma \ref{lemma:lip}, the norm cocycle $\Phi(g, x)$ satisfies the Lipschitz bound:
$$
\left| \Phi(g, x) - \Phi(g', x') \right| \leq \sqrt{L_K} \, d_G(g, g') + \sqrt{2} L_K \, d_X(x, x'),
$$
for any $g, g' \in K$ and $x, x' \in X$. Consequently, the drift function $\phi_\mu$ is Lipschitz continuous on $X$ with $\mathrm{Lip}(\phi_\mu) \le \sqrt{2} L_K$.
\end{lemma}

\begin{proof}
For any $g, g' \in K$ and $x, x' \in X$, we can bound the difference using the triangle inequality:
$$
\left| \Phi(g, x) - \Phi(g', x') \right| \leq \left| \Phi(g, x) - \Phi(g', x) \right| + \left| \Phi(g', x) - \Phi(g', x') \right|.
$$

For the first term, let the projective class $x \in X$ be represented by a unit vector $u \in \mathbb{R}^d$. We evaluate the difference with respect to the group variable:
$$
|\Phi(g, x) - \Phi(g', x)| = \left| \log \frac{\|g u\|}{\|g' u\|} \right| \leq \frac{|\|g u\| - \|g' u\||}{\min(\|g u\|, \|g' u\|)} \leq \frac{\|(g - g')u\|}{\min(\|g u\|, \|g' u\|)}.
$$
For any matrix $h \in K$, we have the lower bound $\|hu\| \geq \|h^{-1}\|^{-1} \ge (\sup_{h \in K} \|h^{-1}\|)^{-1}$. Applying this to the denominator yields:
$$
|\Phi(g, x) - \Phi(g', x)| \leq \left( \sup_{h \in K} \|h^{-1}\| \right) \|(g - g')u\| \leq \sqrt{L_K} \|g - g'\|.
$$

For the second term, let $x, x' \in X$ be represented by unit vectors $u, v \in \mathbb{R}^d$. By choosing the appropriate signs for $u$ and $v$, we ensure the angle $\theta$ between them is in $[0, \pi/2]$. The Euclidean distance relates to the angle via $\|u - v\| = \sqrt{2 - 2\cos\theta} = 2 \sin(\theta/2)$. For angles in this range, we have the geometric bound $\|u - v\| = 2 \sin(\theta/2) \leq \sqrt{2} \sin \theta = \sqrt{2} d_X(x, x')$. Evaluating the absolute difference of the cocycle for the fixed matrix $g'$:
$$
|\Phi(g', x) - \Phi(g', x')| = \left| \log \frac{\|g'u\|}{\|g'v\|} \right| \leq \frac{|\|g'u\| - \|g'v\||}{\min(\|g'u\|, \|g'v\|)} \leq \frac{\|g'(u - v)\|}{\|g'^{-1}\|^{-1}} \leq \|g'\| \|g'^{-1}\| \|u - v\|.
$$
Substituting the geometric bound, we find:
$$
|\Phi(g', x) - \Phi(g', x')| \leq \sqrt{2} \|g'\| \|g'^{-1}\| d_X(x, x') \leq \sqrt{2} \sup_{h \in K} N(h)^2 d_X(x, x') = \sqrt{2} L_K d_X(x, x').
$$
Combining these two components establishes the Lipschitz inequality.

Finally, the Lipschitz bound for the drift function follows immediately by setting $g=g'$, and integrating over $G$.
\end{proof}

A probability measure $\nu$ on $X$ is called \emph{$\mu$-stationary} if $P_\mu \nu = \nu$, meaning that for any $\varphi \in \mathscr{C}(X)$,
$$
\int_X \varphi(x) \, d\nu(x) = \int_X \int_G \varphi(gx) \, d\mu(g) \, d\nu(x).
$$
We call a $\mu$-stationary probability measure \emph{maximal} if it attains the maximum in the following Furstenberg's formula. Every $\mu$-stationary probability measure on $X$ is maximal when $\lambda_1(\mu) = \lambda_d(\mu)$.

\begin{thm}[Furstenberg's formula \cite{Fur63}]\label{lemma:formula}
Let $\mu$ be a probability measure on $G$ with finite first moment. Then
$$
\lambda_{1}(\mu) = \max_{\nu} \int_{X} \int_{G} \Phi(g,x) \, d\mu(g) d\nu(x),
$$
where the maximum is taken over all $\mu$-stationary probability measures $\nu$ on $X$.
\end{thm}

\begin{rmk}
The value of $\lambda_1(\mu)$ given by Furstenberg's formula does not depend on the choice of Euclidean norm on $\mathbb{R}^d$ used to define $\Phi$. Indeed, replacing the norm by an equivalent one changes $\Phi(g, x)$ by a function of the form $f(gx) - f(x)$ for some bounded $f$, and the integral of such a coboundary against any stationary measure vanishes. In particular, when $\mu$ is conformal, the adapted norm $\|\cdot\|_c$ may be used in the formula in place of the standard norm.
\end{rmk}

\subsection{The conformal case}
We handle the conformal case separately. 
Denote by $\Gamma_\mu \subset G$ the subsemigroup generated by the support of any $\mu \in \mathscr{P}(G)$.
We call $\mu$ \emph{conformal} if $\mathbb{R}^d$ admits a $\Gamma_\mu$-invariant conformal structure. 
In this case, every $g\in\Gamma_\mu$ acts as a scalar multiple of an isometry, so the norm cocycle $\Phi(g,x)$ is independent of $x$ and equals $\frac{1}{d}\log|\det g|$. The Lyapunov exponents are therefore determined solely by the determinant, which linearizes the dependence on $\mu$ and yields Lipschitz continuity, as the following argument analogous to the one in~\cite{TVi20}.

\begin{prop} \label{prop:conformal}
Let $K \subset G$ be a compact set, and let $\mu, \mu'$ be probability measures on $G$ supported in $K$. Assume $\mu$ is conformal. There exists a constant $C(K,\mu)>0$ such that 
$$
\left| \lambda_{k}(\mu') - \lambda_{k}(\mu) \right| \leq C d_{\mathrm W}(\mu', \mu) ,\quad \text{for every $k \in \{  1, \ldots, d \}$} .
$$
\end{prop}

\begin{proof}
Since $\mu$ is conformal, there exists an adapted Euclidean norm $\| \cdot \|_c$ on $\mathbb{R}^d$ such that every $g \in \Gamma_\mu$ acts as a similarity transformation. Let $\Phi_c(g, x) = \log \frac{\|gx\|_c}{\|x\|_c}$ be the norm cocycle with respect to this adapted norm.

For any $g \in \mathrm{supp}(\mu)$, the similarity factor is independent of $x$, meaning $\Phi_c(g, x) = \frac{1}{d} \log |\det g|$. Consequently, for any probability measure $\nu'$ on $X$, we can write the first Lyapunov exponent of $\mu$ as:
$$
\lambda_1(\mu) = \int_G \frac{1}{d} \log |\det g| \, d\mu(g) = \int_X \int_G \Phi_c(g, x) \, d\mu(g) \, d\nu'(x).
$$

Now let $\nu'$ be a maximal $\mu'$-stationary probability measure. By Furstenberg's formula (Theorem \ref{lemma:formula}):
$$
\lambda_1(\mu') = \int_X \int_G \Phi_c(g, x) \, d\mu'(g) \, d\nu'(x).
$$

Subtracting and using the fact that $\lambda_1(\mu)$ does not depend on $\nu'$, the difference in the first Lyapunov exponents reduces to:
$$
\lambda_1(\mu') - \lambda_1(\mu) = \int_X \left( \int_G \Phi_c(g, x) \, d(\mu' - \mu)(g) \right) d\nu'(x).
$$

For a fixed $x \in X$, the function $g \mapsto \Phi_c(g, x)$ is Lipschitz on the compact set $K$. Because $K$ is bounded and consists of invertible matrices, there exists a constant $C_1(K, \|\cdot\|_c) > 0$ such that $\sup_{x \in X} \mathrm{Lip}(\Phi_c(\cdot, x)) \leq C_1$. 
Thus, the inner integral is bounded directly by $C_1 d_{\mathrm W}(\mu', \mu)$, which gives:
$$
\left| \lambda_1(\mu') - \lambda_1(\mu) \right| \leq C_1 d_{\mathrm W}(\mu', \mu).
$$

For the remaining Lyapunov exponents $k \geq 2$, recall from the previous section that $\lambda_k(\mu) = \lambda_1(\wedge^k \mu) - \lambda_1(\wedge^{k-1} \mu)$. The conformal structure on $\mathbb{R}^d$ canonically induces a conformal structure on the exterior algebraic spaces $\wedge^k \mathbb{R}^d$, meaning the pushforward $\wedge^k \mu$ is also conformal. 

Applying the same argument to the exterior powers gives a constant $C_k'(K, \|\cdot\|_c) > 0$ such that:
$$
\left| \lambda_1(\wedge^k \mu') - \lambda_1(\wedge^k \mu) \right| \leq C_k' d_{\mathrm W}(\wedge^k \mu', \wedge^k \mu).
$$

By Lemma~\ref{lemma:trivial}, $d_{\mathrm W}(\wedge^k \mu', \wedge^k \mu) \leq k L_K^{(k-1)/2} d_{\mathrm W}(\mu', \mu)$. Combining these inequalities yields the uniform Lipschitz bound for all $k \in \{1, \dots, d\}$.
\end{proof}

% ----------------------------------------------------------------------
\section{Algebraic structure of the semigroup}\label{sec:alg}

Proposition~\ref{prop:conformal} handles the case when $\mu$ is conformal. In the non-conformal case, the key idea is that the semisimplicity assumption and the one-point Lyapunov spectrum together force a structured decomposition of $\mathbb{R}^d$ into conformal blocks that are permuted by $\Gamma_\mu$, reducing the problem to a certain convergence rate towards these components of the decomposition. Throughout the rest of the paper, we adopt the convention of using subscripts to denote sequence indices (time steps $n$) and radii of neighborhoods ($r$ and $R$), and superscripts to denote powers, coordinates, or subspace indices ($j \in J$).

Let $V$ be a finite-dimensional real vector space and $\Gamma$ a subsemigroup of $\mathrm{GL}(V)$. A $\Gamma$-invariant subspace $W$ is called $\Gamma$-\textit{irreducible} if it contains no proper non-trivial $\Gamma$-invariant subspace. A subspace $W$ is called \textit{virtually} $\Gamma$-invariant if it has a finite $\Gamma$-orbit $\{ g W : g \in \Gamma \}$. A virtually $\Gamma$-invariant subspace is called \textit{strongly} $\Gamma$-\textit{irreducible} if it contains no proper non-trivial virtually $\Gamma$-invariant subspace. We omit the reference to $\Gamma$ if the context is clear.

\subsection{Virtually conformal decompositions}

\begin{lemma}\label{lemma:decomp}
Let $V$ be a finite-dimensional real vector space and $\Gamma$ be a subsemigroup of $\mathrm{GL}(V)$. Assume the action of $\Gamma$ on $V$ is semisimple. Then there exists a decomposition $V=\oplus _{j \in J}W^{j}$ such that each $W^{j}$ is a virtually invariant and strongly irreducible subspace. If further, the action of $\Gamma$ on $V$ is irreducible, then $\Gamma$ permutes the set $\left\{ W^{j} \right\}_{j\in J}$ transitively.
\end{lemma}

\begin{proof}
$V_1=V$ is virtually invariant. Since dimension decreases with proper subspaces, a minimal non-trivial virtually invariant subspace $W$ exists. Consider the finite $\Gamma$-orbit $\left\{ W^{j} \right\}_{j \in J^1}$ of $W$. The sum $\sum_{j\in J^1}W^{j}$ is invariant. Minimality means that every $W^{j}$ is strongly irreducible. It follows that their sum is direct and their sum subspace is irreducible. Let $V_2$ be a complementary invariant subspace of $\oplus_{j\in J^1} W^j$ by semisimplicity and we obtain another irreducible invariant subspace $\oplus_{j\in J^2}W^j \subset V_2$ in the same way. This induction ends when we have $V=\oplus_{j\in J}W^j$, where $J=\cup_{i\in I} J^i$. Transitivity follows from the definition of an orbit when the action is irreducible.
\end{proof}

We denote by $S_J$ the permutation group of a finite set $J$ that indexes a $\Gamma$-invariant family of subspaces, and by $F$ the image of $\Gamma$ in $S_J$ under the map $g \mapsto \pi$ with $W^{\pi(j)}=gW^j$ for every $j \in J$.

\begin{defn} \label{def:decomp}
Let $V$ be a finite-dimensional real vector space and $\Gamma$ be a subsemigroup of $\mathrm{GL}(V)$. A decomposition $V=\oplus _{j \in J} W^{j}$ into virtually $\Gamma$-invariant subspaces is called \emph{virtually $\Gamma$-conformal} if there exist conformal structures $[\mathcal{C}^{j}]$ on $W^{j}$ for every $j\in J$ such that: for any $g \in \Gamma$, there exists a permutation $\pi\in S_J$, such that $gW^{j} = W^{\pi(j)}$ and $g_{*}[\mathcal{C}^{j}]=[\mathcal{C}^{\pi(j)}]$, for any $j \in J$.

In this case, $g$ preserves the decomposition up to permutation and is conformal on each summand. One can choose a basis of $\mathbb{R}^{d}$ so that every $g \in \Gamma$ is a product $g=qa$ of a block permutation matrix $q$ and a block conformal matrix $a$:
$$
g =
\left(
\begin{array}{ccccccccc}
&  & & & & q^{j'} & & & \\
&  & \dots & & & & & & \\
 & & & & & & & & \\
 & & & & & & & & \\
q^{1}& & & & \dots & & & & \\
& & & & & & & & \\
 & & & & & & & & q^{\left| J \right| } \\
& & & & & & \dots & & \\
& & & q^j & & & &  &
\end{array}
\right)
 \left(
\begin{array}{ccccccccc}
a^{1} & & & & & & & & \\
& \ddots & & & & & & & \\
& & a^{j} & & & & & & \\
 & & & & & & & & \\
& & & & \ddots & & & & \\
& & & & & & & & \\
 & & & & & & a^{j'} & & \\
& & & & & & & \ddots & \\
& & & & & & & & a^{\left| J \right|}
\end{array}
\right),
$$
where every $a^{j}$ is a scalar matrix and every $q^{j}$ is an orthogonal matrix. We also denote the scalar by $a^{j}>0$. Note that this basis can be adjusted by a block-conformal change of coordinates $T = \oplus_{j\in J} t^j \mathrm{Id}_{W^j}$ with $t^j>0$ for all $j\in J$.

The virtually $\Gamma$-conformal decompositions form a partially ordered set with the following partial order: $\oplus_{i \in I} U^i$ is \emph{coarser} (smaller) than $\oplus_{j \in J} W^j$ if (i) for any $j\in J$ there exists $i \in I$ such that $U^i \supset W^j$, and (ii) there exist conformal structures $[\mathcal{D}^i]$ on $U^i$ and $[\mathcal{C}^j]$ on $W^j$, such that $\mathcal{D}^i |_{W^j} = \mathcal{C}^j$ (as norms) for all $i\in I$ and $j \in J$ with $U^i \supset W^j$.
\end{defn}

\begin{lemma} \label{lemma:coarsest}
Let $V$ be a finite-dimensional real vector space and $\Gamma$ be a subsemigroup of $\mathrm{GL}(V)$. There exists a coarsest virtually $\Gamma$-conformal decomposition, as long as $V$ admits one virtually $\Gamma$-conformal decomposition.
\end{lemma}
\begin{proof}
Since the number of direct summands decreases with strictly coarser decompositions, a minimal virtually $\Gamma$-conformal decomposition strictly exists.
\end{proof}

\begin{lemma} \label{lemma:uniq_conf}
Let $V$ be a finite-dimensional real vector space and $\Gamma$ be a subsemigroup of $\mathrm{GL}(V)$. Assume the action of $\Gamma$ on $V$ is irreducible and preserves a conformal structure. Then this conformal structure is the unique $\Gamma$-invariant one.
\end{lemma}
\begin{proof}
We may represent the two conformal structures by positive-definite symmetric bilinear forms $Q_{1}$ and $Q_{2}$. For any $h \in \Gamma$, conformal invariance implies there exist scalars $\eta_{1}(h)>0$ and $\eta_{2}(h)>0$ such that
$$
h^{T}Q_{1}h=\eta_{1}(h)Q_{1} \quad \text{and} \quad h^{T}Q_{2}h=\eta_{2}(h)Q_{2}.
$$
Notice that $\eta_{1}$ and $\eta_{2}$ are semigroup homomorphisms from $\Gamma$ to $\Rpos$.

Consider the isomorphism $A=Q_{1}^{-1}Q_{2}$. Because $Q_{1}$ and $Q_{2}$ are positive-definite and symmetric, the eigenvalues of $A$ are all strictly positive real numbers. We compute the conjugation of $A$ by $h$:
$$
h^{-1} A h = h^{-1} Q_{1}^{-1} (h^{T})^{-1} h^{T} Q_{2} h = (h^{T}Q_{1}h)^{-1}(h^{T}Q_{2}h).
$$
Substituting the conformal relations yields:
$$
h^{-1} A h = (\eta_{1}(h)Q_{1})^{-1}(\eta_{2}(h)Q_{2}) = \frac{\eta_{2}(h)}{\eta_{1}(h)} Q_{1}^{-1}Q_{2} = \frac{\eta_{2}(h)}{\eta_{1}(h)} A.
$$
This equation implies that if $v$ is an eigenvector of $A$ with eigenvalue $\beta$, then $A(hv) = h(h^{-1}Ah)v = \frac{\eta_{2}(h)}{\eta_{1}(h)} \beta (hv)$. Thus, $\Gamma$ permutes the eigenspaces of $A$, and acts on the set of eigenvalues by multiplying them by the factor $\frac{\eta_{2}(h)}{\eta_{1}(h)}$.

Since the set of eigenvalues of $A$ is a finite subset of $\Rpos$, the only way a uniform positive scaling can preserve this set is if the scaling factor is exactly $1$. Therefore, for any $h \in \Gamma$, we must have $\eta_{1}(h)=\eta_{2}(h)$.

This implies that $h^{-1} A h = A$, meaning $Ah = hA$ for all $h \in \Gamma$. Because $A$ commutes with $\Gamma$, every eigenspace of $A$ is a $\Gamma$-invariant subspace. Since $A$ has at least one real eigenvalue, its corresponding eigenspace is non-trivial. By the assumption that the action of $\Gamma$ is irreducible, this eigenspace must be the entire space. Consequently, $A = \beta\,\mathrm{Id}$ for some $\beta>0$. This gives $Q_{2} = \beta Q_{1}$, proving that the two conformal structures are identical: $[Q_{1}]=[Q_{2}]$.
\end{proof}

Denote the return subsemigroup of $\Gamma$ to a subspace $W$ by $\Gamma_W = \{ g \in \Gamma: gW=W \}$. The following lemma ensures the existence of one virtually $\Gamma$-conformal decomposition.

\begin{lemma}\label{lemma:admits}
Let $V$ be a finite-dimensional real vector space and $\Gamma$ be a subsemigroup of $\mathrm{GL}(V)$. Assume the action of $\Gamma$ on $V$ is semisimple and there exists a $\Gamma_{W}$-invariant conformal structure on $W$ for any strongly irreducible virtually invariant subspace $W$. Then $V$ admits a virtually $\Gamma$-conformal decomposition.
\end{lemma}
\begin{proof}
Let $\oplus_{j\in J}W^j$ be the decomposition given by Lemma \ref{lemma:decomp}. By the assumption, there exist conformal structures $[\mathcal{C}^j]$ on $W^j$ for all $j\in J$, such that $g_*[\mathcal{C}^j]=[\mathcal{C}^j]$, for any $g \in \Gamma_{W^j}$. 

Now let $g \in \Gamma, j \in J$ and $j' = \pi(j)$. Consider the pushforward conformal structure $g_*[\mathcal{C}^j]$ on $W^{j'}$. It is preserved by $g \Gamma_{W^j} g^{-1}$. However, $g^{-1}$ may not live in $\Gamma$, so we may not have $g \Gamma_{W^j} g^{-1} \supset \Gamma_{W^{j'}}$.
Consider the Zariski closure $H$ of $\Gamma$, which is a group (see for example \cite{BeQ16}). Denote similarly the return subgroup of $H$ by $H_W=\{ h \in H : hW=W \}$ for any subspace $W$.
A conformal structure on $W$ is preserved by $\Gamma_W$ if and only if it is preserved by $H_W$. We know that $g_*[\mathcal{C}^j]$ is preserved by $g H_{W^j} g^{-1} = H_{W^{j'}}$, hence it is also preserved by $\Gamma_{W^{j'}}$. 
Lemma \ref{lemma:uniq_conf} ensures the uniqueness of the invariant conformal structure, yielding $g_*[\mathcal{C}^j] = [\mathcal{C}^{j'}]$. We therefore have a virtually $\Gamma$-conformal decomposition $\oplus_{j\in J}W^j$.
\end{proof}

\subsection{Application of Furstenberg's criterion}
The semigroup $\Gamma_\mu$ satisfies the structural assumptions due to Furstenberg's criterion (Theorem \ref{lemma:criterion}) and the behavior of Lyapunov exponents under induced transformations.

\begin{lemma}[Return Lyapunov exponents] \label{lemma:return}
Let $\mu$ be a probability measure on $G$ with a finite first moment. Let $W$ be a virtually $\Gamma_\mu$-invariant subspace and $\mu_W$ be the pushforward of $\mu^{\mathbb{N}^*}$ by the return map 
\begin{equation}
(g_1, \ldots, g_n, \ldots) \mapsto g_{\tau} \cdots g_2 g_1,
\end{equation}
where $\tau \geq 1$ is the minimal time such that $g_\tau \cdots g_2 g_1$ preserves $W$. Then $\mu_W$ has a finite first moment and 
\begin{equation} \label{eq:le}
\lambda_1 (\mu_W) = \#\{ g W : g \in \Gamma_\mu \} \cdot \lambda_1(\mu). 
\end{equation}
\end{lemma}

\begin{proof}
Let $B = G^{\mathbb{N}^*}$ be the space of trajectories $b = (g_1, g_2, \dots)$ equipped with the product measure $\mathbb{P} = \mu^{\otimes \mathbb{N}^*}$ and the left shift operator $S$. The semigroup $\Gamma_\mu$ acts transitively on the finite orbit of virtually invariant subspaces $\mathcal{O} = \{gW : g \in \Gamma_\mu\}$. Let $M = \#\mathcal{O}$.

Consider the skew-product dynamical system $Y = B \times \mathcal{O}$ with the transformation $T(b, U) = (S b, g_1 U)$. Because the action of $\Gamma_\mu$ on $\mathcal{O}$ is transitive, $T$ preserves the measure $\mathbb{P} \times \mathbf{u}$, where $\mathbf{u}$ is the uniform probability measure on $\mathcal{O}$.

Let $A = B \times \{W\} \subset Y$. The measure of this set is $(\mathbb{P} \times \mathbf{u})(A) = 1/M$. The first return time of a point $(b, W) \in A$ to $A$ under $T$ is exactly $\tau(b)$, the minimal time such that $g_{\tau(b)} \dots g_1 W = W$. 
Let $T_A$ be the induced transformation on $A$, which preserves the conditional probability measure $\mathbb{P}_A = (\mathbb{P} \times \mathbf{u})(\cdot \mid A)$.

To prove $\mu_W$ has a finite first moment, define the function $\varphi: Y \to \mathbb{R}$ by $\varphi(b, U) = \log N(g_1)$. Since $\mu$ has a finite first moment, $\varphi \in L^1(Y, \mathbb{P} \times \mathbf{u})$. By the submultiplicativity of $N$, the return blocks $h = g_{\tau(b)} \dots g_1$ satisfy
$$
\log N(h) \leq \sum_{l=0}^{\tau(b)-1} \log N(g_{l+1}) = \sum_{l=0}^{\tau(b)-1} \varphi(T^l(b, W)).
$$
By Birkhoff's Ergodic Theorem (specifically Kac's integral formula for induced transformations), the induced sum $\sum_{l=0}^{\tau(b)-1} \varphi(T^l(b, W))$ is integrable with respect to $\mathbb{P}_A$, and its integral is exactly 
$$
\frac{1}{(\mathbb{P} \times \mathbf{u})(A)} \int_Y \varphi \, d(\mathbb{P} \times \mathbf{u}) = M \int_G \log N(g) \, d\mu(g) < \infty.
$$
This bounds the first moment of $\mu_W$.

Let $\tau_m$ be the time of the $m$-th return to $W$. By Birkhoff's Ergodic Theorem applied to the return time function $\tau(b)$ on the space $(A, T_A, \mathbb{P}_A)$, we have almost surely
$$
\lim_{m \to \infty} \frac{\tau_m}{m} = \int_A \tau \, d\mathbb{P}_A = \frac{1}{(\mathbb{P} \times \mathbf{u})(A)} = M.
$$
Finally, the product of the first $m$ blocks distributed according to $\mu_W$ is exactly the product of the first $\tau_m$ original matrices: $h_m \dots h_1 = g_{\tau_m} \dots g_1$. By definition, the first Lyapunov exponent for $\mu_W$ is given by
$$
\lambda_1(\mu_W) = \lim_{ m \to \infty } \frac{1}{m} \log \sigma_1(h_m \dots h_1) = \lim_{ m \to \infty } \frac{\tau_m}{m} \frac{1}{\tau_m} \log \sigma_1(g_{\tau_m} \dots g_1).
$$
By Furstenberg-Kesten theorem (or Kingman's subadditive ergodic theorem), the sequence \\ $\frac{1}{n} \log \sigma_1(g_n \dots g_1)$ converges to $\lambda_1(\mu)$ almost surely. Because $\tau_m \to \infty$ almost surely, the subsequence converges to the exact same limit:
$$
\lim_{ m \to \infty } \frac{1}{\tau_m} \log \sigma_1(g_{\tau_m} \dots g_1) = \lambda_1(\mu).
$$
Multiplying the limits yields $\lambda_1(\mu_W) = M \cdot \lambda_1(\mu)$.
\end{proof}

\begin{rmk}
If $W$ is strongly $\Gamma_\mu$-irreducible, then $\#\{ g W : g \in \Gamma_\mu \} = \frac{d}{\mathrm{dim}(W)}$. 
\end{rmk}

\begin{rmk}
For a general compactly supported probability measure $\mu$, the induced probability measure $\mu_W$ would not be compactly supported. However, this should not worry us because we only use formula (\ref{eq:le}) to describe the structure of the semigroup $\Gamma_\mu$ (Proposition \ref{prop:struc_assump}); the measure $\mu_W$ will not be used in any other way.
\end{rmk}

\begin{thm}[Furstenberg \cite{Fur63}]\label{lemma:criterion}
Let $\mu$ be a probability measure on $G$ with finite first moment. If $\lambda_{1}(\mu)=\lambda _{d}(\mu)$ and the action of $\Gamma _{\mu}$ on $\mathbb{R}^{d}$ is strongly irreducible, then $\Gamma _{\mu}$ preserves a conformal structure on $\mathbb{R}^{d}$.
\end{thm}

Assembling Furstenberg's criterion (Theorem \ref{lemma:criterion}) with Lemma~\ref{lemma:admits} and Lemma~\ref{lemma:coarsest} gives the main structural result used in the proof of Theorem~\ref{thmA}.

\begin{prop} \label{prop:struc_assump}
Let $\mu$ be a compactly supported probability measure on $G$. Assume $\mu$ is semisimple and its Lyapunov spectrum satisfies $\lambda_1(\mu) = \lambda_d(\mu)$. Let $\Gamma_\mu$ be the subsemigroup generated by the support of $\mu$. Then $\mathbb{R}^d$ admits a coarsest virtually $\Gamma_\mu$-conformal decomposition. 
\end{prop}
\begin{proof}
Let $W$ be any strongly $\Gamma_\mu$-irreducible virtually $\Gamma_\mu$-invariant subspace. Consider the induced probability measure $\mu_W$ on $\mathrm{GL}(W)$. The subsemigroup generated by $\mathrm{supp}(\mu_W)$ is
$$
\Gamma _{\mu _W} = \left\{ g \in \Gamma _{\mu}  : gW = W \right\} = \left( \Gamma_{\mu} \right)_W.
$$
Since $\lambda_{1}(\mu _{W})=\lambda_{\mathrm{dim}(W)}(\mu _{W})$ (by Lemma \ref{lemma:return}) and $\Gamma _{\mu_W}$ is strongly irreducible on $W$, $(\Gamma_\mu)_{W}$ preserves a conformal structure on $W$ by Furstenberg's criterion (Theorem \ref{lemma:criterion}). Lemmas \ref{lemma:admits} and \ref{lemma:coarsest} finish the proof.
\end{proof}

\subsection{Aperiodicity on finite groups}

A probability measure on a finite group $F$ is called \emph{aperiodic} if its support generates $F$ as a semigroup and its (pushforward) image is not a Dirac mass under any non-trivial group homomorphism to a cyclic group. 

Recall that we denote by $F$ the image of $\Gamma$ through a permutation on a $\Gamma$-invariant family of subspaces.
The $\Gamma_\mu$-invariant family of subspaces is taken to be the components of the coarsest virtually $\Gamma_\mu$-conformal decomposition, for a semisimple probability measure $\mu$ on $G$ with a one-point Lyapunov spectrum.
Then $F_\mu$ is the image of $\Gamma_\mu$ through the permutation on these decomposition components.
We slightly abuse the notation $\mu$: the probability measure on $F_\mu$ induced by $\mu$ is also denoted by $\mu$.

The following two lemmas regarding aperiodicity and the exponential convergence of the random walk can be found, for example, in Benoist and Quint \cite{BeQ16}. For completeness, we provide their proofs based on standard algebraic and spectral arguments.

\begin{lemma}[Aperiodicity] \label{lemma:aper}
Let $\mu$ be a probability measure on a finite group $F$ whose support generates $F$. There exists a normal subgroup $A$ of $F$, such that the convolution power $\mu ^{*p}$ is aperiodic in $A$, where $p = \left| F / A \right|$.
\end{lemma}

\begin{proof}
Let $S = \mathrm{supp}(\mu)$. By assumption, the semigroup generated by $S$ is the entire finite group $F$. Define $A$ to be the subgroup of $F$ generated by all differences between elements in the support:
$$
A = \langle x^{-1}y : x, y \in S \rangle.
$$
To see that $A$ is a normal subgroup, let $F(S)$ be the free group generated by the set $S$, and let $\pi: F(S) \to F$ be the natural evaluation homomorphism. Because $S$ generates $F$, $\pi$ is surjective. 

Now, define a homomorphism $\psi: F(S) \to \mathbb{Z}$ that maps every generator $s \in S$ to $1$. The kernel of this map, $\ker(\psi)$, consists of all words in $F(S)$ with an exponent sum of zero, which is classically generated by elements of the form $x^{-1}y$ for $x, y \in S$. Since $\ker(\psi)$ is a normal subgroup of $F(S)$ and $\pi$ is surjective, its image $\pi(\ker(\psi))$ must be a normal subgroup of $F$. By definition, this image is exactly $A$, proving $A \triangleleft F$.

Because $x^{-1}y \in A$ for all $x, y \in S$, we have $x A = y A$. Therefore, the entire support $S$ is contained in a single coset of $A$, say $c A$ for some $c \in S$. Since $S$ generates $F$, the quotient group $F / A$ is generated by the single element $c A$. This implies $F / A$ is a cyclic group. Let $p = |F / A|$ be its order.

Now consider the $p$-th convolution $\mu^{*p}$. Its support is $S^{p}$, which consists of products of $p$ elements from $S$. In the quotient group $F / A$, any element in $S^{p}$ projects to $(c A)^{p} = A$, the identity element. Thus, $S^{p} \subset A$, meaning $\mu^{*p}$ is a probability measure supported on $A$.

To see that $\mu^{*p}$ is aperiodic in $A$, we must verify two conditions: first, that the support of $\mu^{*p}$ generates $A$ as a semigroup, and second, that its image is not a Dirac mass under any non-trivial group homomorphism from $A$ to a cyclic group.

For the first condition, let $S_{p} = S^{p}$ be the support of $\mu^{*p}$. Since $S$ generates the finite group $F$, any subsemigroup it generates is also a subgroup. Let $A'$ be the subgroup of $A$ generated by $S_{p}$. Fix an element $c \in S$. Clearly, $c^{p} \in S_{p} \subset A'$. For any $x \in S$, the product $x c^{p-1}$ consists of $p$ elements of $S$, so $x c^{p-1} \in S_{p} \subset A'$. Consequently, the element $(x c^{p-1})(c^{p})^{-1} = x c^{-1}$ must also belong to $A'$. Because the normal subgroup $A$ is generated by elements of the form $x^{-1}y$ for $x, y \in S$, and any such element can be rewritten as $(x c^{-1})^{-1}(y c^{-1})$, it follows that $A'$ contains all generators of $A$. Thus, $A' = A$, meaning the support of $\mu^{*p}$ generates $A$.

For the second condition, suppose for the sake of contradiction that there exists a group homomorphism $\phi: A \to \mathbb{Z}/k\mathbb{Z}$ such that the pushforward $\phi_*(\mu^{*p})$ is a Dirac mass. This implies there exists a fixed element $z \in \mathbb{Z}/k\mathbb{Z}$ such that $\phi(u) = z$ for all $u \in S_{p}$. Using the same $c \in S$ and arbitrary $x \in S$ as above, we evaluate the homomorphism on their products: $\phi(c^{p}) = z$ and $\phi(x c^{p-1}) = z$. Since $A$ acts transitively on these quotients, we can take the difference in the cyclic group:
$$
\phi(c^{-1}x) = \phi((c^{p})^{-1} x c^{p-1}) = -z + z = 0.
$$
Since this holds for any $x \in S$, we evaluate the generators of $A$: for any $x, y \in S$,
$$
\phi(x^{-1}y) = \phi((c^{-1}x)^{-1} (c^{-1}y)) = 0 + 0 = 0.
$$
Because the elements $x^{-1}y$ generate the entire subgroup $A$, we must have $\phi(a) = 0$ for all $a \in A$. This establishes that $\phi$ is the trivial homomorphism. Therefore, $\mu^{*p}$ cannot map to a Dirac mass under any non-trivial homomorphism to a cyclic group, concluding the proof that $\mu^{*p}$ is aperiodic in $A$.
\end{proof}

\begin{lemma} \label{lemma:markov}
Let $\mu$ be an aperiodic probability measure on a finite group $F$. There exist constants $C > 0$ and $\rho \in (0,1)$ such that for any $n \in \mathbb{N}^*$ and any $s \in F$,
$$
\left| \mu^{*n}(\{ s \}) - \frac{1}{|F|} \right| \leq C \rho^n.
$$
\end{lemma}

\begin{proof}
The sequence of random variables defined by the products $g_n \dots g_1$ induces a left random walk on the finite group $F$. Because left multiplication by any group element acts as a bijection on $F$, the transition matrix $P$ is doubly stochastic. Consequently, the uniform distribution $s \mapsto \frac{1}{|F|}$ is preserved under this action and serves as a stationary measure for the Markov chain.

Since $\mu$ is an aperiodic probability measure whose support generates $F$, the random walk is both irreducible and aperiodic. By the Perron-Frobenius theorem for finite-state Markov chains, $P$ has a simple eigenvalue at $1$, corresponding to the constant functions, and all other eigenvalues strictly reside within the open unit disk.

Let $V_0 = \{\varphi: F \to \mathbb{R} \mid \sum_{s \in F} \varphi(s) = 0\}$ be the subspace of functions with zero mean. The transition matrix $P$ preserves $V_0$. Because the state space is finite, the spectral radius $\rho$ of $P$ restricted to $V_0$ is strictly less than $1$. Consequently, we can choose a constant $C > 0$ accounting for the eigenvectors, such that for any $n\in \mathbb{N}^*$,
$$
\sup_{s \in F} \left| \mu^{*n}(\{s\}) - \frac{1}{|F|} \right| \leq C \rho^n.
$$
\end{proof}

% ----------------------------------------------------------------------
\section{Reduction to empirical measures}\label{sec:reduction}

Having established the algebraic structure of $\Gamma_\mu$ in Section~\ref{sec:alg}, we now translate the problem of bounding $|\lambda_1(\mu') - \lambda_1(\mu)|$ into one of controlling how quickly the iterates of the Markov operator $P_\mu$ push an arbitrary probability measure toward the $\mu$-stationary measures.

From this point on we work in the non-conformal case, so we assume the coarsest virtually $\Gamma_\mu$-conformal decomposition is non-trivial: $\mathbb{R}^d = \oplus_{j \in J} W^j$ with $|J| \ge 2$. Let $F_\mu$ be the image of $\Gamma_\mu$ in $S_J$, and let $J = \cup_{i\in I} J^i$ be the partition of $J$ into $F_\mu$-orbits.

\subsection{Iterating the Markov operator}

\begin{lemma}\label{lemma:estimate1}
Let $K \subset G$ be a compact set, and let $\mu$ and $\mu'$ be probability measures on $G$ supported in $K$. Let $\nu'$ be any maximal $\mu'$-stationary probability measure on $X$. Define the iterates $\nu_n = P_\mu^n \nu'$ for $n \in \mathbb{N}$. There exists a constant $C_K \geq 2$ such that
$$
\left| \lambda_{1}(\mu') - \lambda_{1}(\mu) \right| \leq C_K^{n} d_{\mathrm W}(\mu', \mu) + \left| \int_X \phi _{\mu} \, d \nu _{n} - \lambda_{1}(\mu) \right|.
$$
\end{lemma}

\begin{proof}
By Furstenberg's formula (Theorem \ref{lemma:formula}) and the triangle inequality,
\begin{align*}
\left| \lambda_{1}(\mu') - \lambda_{1}(\mu) \right| &\leq \left|  \int_{G\times X} \Phi \, d\mu'\otimes\nu' -  \int_{G\times X} \Phi \, d \mu \otimes \nu' \right| \\
&+ \left| \int_X \phi_\mu \, d\nu' - \int_X \phi_\mu \, d\nu _{n} \right|
+ \left| \int_X \phi_\mu \, d\nu _{n} - \lambda_{1}(\mu) \right|.
\end{align*}
The first term is bounded by the Wasserstein distance $d_{\mathrm W}(\mu',\mu)$ multiplied by the Lipschitz constant of the function $g\mapsto\int_X \Phi(g,x)\,d\nu'(x)$ on the compact support:
$$
\mathrm{Lip}\left( \int_X \Phi(\cdot,x) \, d\nu'(x) \right)  \leq
\sup_{x \in X} \mathrm{Lip}(\Phi(\cdot, x)) \leq \sqrt{L_K}.
$$
Notice that because $\nu'$ is $\mu'$-stationary, we have $P_{\mu'}\nu'=\nu'$, and we can rewrite the difference as a telescopic sum:
$$
\nu' - \nu_n = \sum_{m=1}^n (\nu_{m-1} - \nu_m).
$$
With $\nu'=\nu_{0}$, we bound the second term by dualizing the operators back onto the drift function by virtue of Lemma \ref{lemma:lip} (ii):
\begin{align*}
\left| \int_X \phi_\mu \, d\nu' - \int_X \phi_\mu \, d \nu _{n} \right| &\leq \sum_{m=1}^{n} \left| \int_X P^{m-1}_{\mu}\phi_\mu \, d (P_{\mu'} \nu' - P_{\mu}\nu') \right| \\
&\leq \sum_{m=1}^{n} \sqrt{L_K} \, \mathrm{Lip}(P^{m-1}_{\mu}\phi_\mu) \, d_{\mathrm W}(\mu',\mu).
\end{align*}
By Lemma \ref{lemma:lip} (i) and Lemma \ref{lemma:drift_lip}, the iterated Lipschitz bound scales multiplicatively by $L_K^{m-1}$ such that:
$$
\mathrm{Lip}(P^{m-1}_{\mu}\phi_\mu) \leq L_K^{m-1}\mathrm{Lip}(\phi_\mu) \leq \sqrt{2} L_K^{m}. 
$$
Substituting this bound yields:
\begin{equation*} \label{eqn:second}
\left| \int_X \phi_\mu \, d\nu' - \int_X \phi_\mu \, d\nu _{n} \right| \leq \sum_{m=1}^n \sqrt{2} L_K^{m+1/2} d_{\mathrm{W}}(\mu', \mu) \leq \sqrt{2} n L_K^{n+1/2} d_{\mathrm W}(\mu',\mu).
\end{equation*}
Adding the bounds on the first two terms, and setting $C_K = 2 L_K^{2}$, the cumulative bound $\sqrt{L_K} + \sqrt{2} n L_K^{n+1/2}$ is dominated by $C_K^{n}$ for $n \ge 1$.
\end{proof}

\subsection{Partitioning the space}

The decomposition $\mathbb{R}^d = \oplus_{j\in J} W^j$ structures the projective space $X$ into regions where the drift function $\phi_\mu$ is nearly constant. The following two lemmas make this precise and provide the key decomposition of $|\int_X \phi_\mu\,d\xi - \lambda_1(\mu)|$ used in the proof of Proposition~\ref{prop:conv}.

\begin{lemma}\label{lemma:const}
Let $\mu$ be a compactly supported semisimple probability measure on $G$ with $\lambda_1(\mu) = \lambda_d(\mu)$, and let $\mathbb{R}^d=\oplus_{j\in J}W^j$ be a virtually $\Gamma_\mu$-conformal decomposition. The function $\phi _{\mu}$ is constant on every $\mathbb{P}(W^{j})$, for $j\in J$, and $\lambda_{1}(\mu)$ is the average,
$$
\lambda_{1}(\mu) = \frac{1}{\left| J^{i} \right|}\sum_{j \in J^{i}} \phi_\mu |_{\mathbb{P}(W^{j})} ,
$$
for all $i \in I$.
\end{lemma}

\begin{proof}
For any $\mathbb{R}v^{j} \in \mathbb{P}(W^{j})$, 
\begin{align*}
\phi _{\mu}(\mathbb{R}v^{j}) &= \int_G \log{ \frac{\|gv^{j}\|}{\|v^{j}\|}} \, d \mu(g) \\
& =  \int_G  \log{\frac{\| q^{j}(g) a^{j}(g)v^{j}\|}{\|v^{j}\|} } \, d \mu(g) \\
& = \int_G  \log{ a^{j}(g) }  \, d \mu(g) .
\end{align*}
The uniform probability measures $\nu ^{i} := \frac{1}{\left| J^{i} \right|} \sum_{j \in J^{i}} \mathrm{Vol}_{\mathbb{P}(W^{j})}$ are $\mu$-stationary for all $i \in I$, thus capturing the top exponent: 
$$
\lambda_{1}(\mu) = \int_X \int_G \Phi(g,x) \, d\mu(g) \otimes \nu ^{i}(x) = \frac{1}{\left| J^{i} \right|} \sum_{j\in J^{i}} \phi _{\mu}|_{\mathbb{P}(W^{j})}.
$$
\end{proof}

\begin{lemma}\label{lemma:estimate2}
Let $\mu$ be a compactly supported semisimple probability measure on $G$ with $\lambda_1(\mu) = \lambda_d(\mu)$. There exists a constant $L_K>0$ depending only on the compact support $K$ of $\mu$, such that for any probability measure $\xi$ on $X$ and any $r\in (0, \frac{1}{2})$,
$$
\left| \int_X \phi _{\mu} \, d\xi - \lambda_{1}(\mu) \right| \leq \sqrt{2} L_K \left( r + \xi(B_{r}^{c}) + \sum_{i \in I} \sum_{j,j' \in J^{i}} \left| \xi(B^{j'}) - \xi(B^{j}) \right| \right) ,
$$
where $B^{j}$ are $r$-neighborhoods of $\mathbb{P}(W^j)$ and $B^{c}_r$ is the complement $X \setminus \cup_{j\in J}B^{j}$ of their union.
\end{lemma}

\begin{proof}
The neighborhoods $B^{j}$ are pairwise disjoint for $r\in\left( 0, \frac{1}{2} \right)$. By splitting the domain of integration,
\begin{align*}
\int_X \phi_\mu \, d\xi &= \sum_{i \in I} \sum_{j \in J^{i}} \int _{B^{j}} \phi_\mu \, d\xi + \int _{B^{c}_{r}} \phi_\mu \, d\xi \\
&= \sum_{i \in I} \sum_{j \in J^{i}} \int _{B^{j}} \left( \phi_\mu - \phi_\mu|_{\mathbb{P}(W^{j})} \right) \, d\xi + \int _{B_{r}^{c}} \phi_\mu \, d\xi + \sum_{i\in I} \sum_{j \in J^{i}} \phi_\mu|_{\mathbb{P}(W^{j})} \cdot \xi(B^{j}) .
\end{align*}
For the third term, let $c^j = \phi_\mu|_{\mathbb{P}(W^j)} - \lambda_1(\mu)$ and let $p^i = \sum_{j \in J^i} \xi(B^j)$ be the total mass of $\xi$ on the $r$-neighborhood of the $i$-th orbit $\cup_{j\in J^i} \mathbb{P}(W^j)$. By Lemma \ref{lemma:const}, we have $\sum_{j \in J^i} c^j = 0$ for each orbit $i \in I$. Fixing a reference index $j^i \in J^i$ for each orbit, we rewrite the sum over $J^i$ using the cancellation $\sum_{j \in J^i} c^j = 0$:
$$ 
\sum_{j \in J^i} c^j \left( \xi(B^j) - \frac{p^i}{|J^i|} \right) = \sum_{j \in J^i} c^j \left( \xi(B^j) - \xi(B^{j^i}) \right), 
$$
where we used $\sum_{j \in J^i} c^j \xi(B^{j^i}) = \xi(B^{j^i}) \sum_{j \in J^i} c^j = 0$. This gives:
$$ 
\left| \sum_{j \in J^i} c^j \left( \xi(B^j) - \frac{p^i}{|J^i|} \right) \right| \leq \max_{j \in J^i} |c^j| \sum_{j \in J^i} \left| \xi(B^j) - \xi(B^{j^i}) \right|. 
$$
We therefore decompose the total difference as:
\begin{align*}
\int_X \phi_\mu \, d\xi - \lambda_{1}(\mu) &= \sum_{i \in I, j \in J^{i}} \int _{B^{j}} \left( \phi_\mu - \phi_\mu|_{\mathbb{P}(W^{j})} \right) \, d\xi + \int_{B^{c}_{r}} (\phi_\mu - \lambda_{1}(\mu)) \, d\xi \\
&+ \sum_{i \in I} \sum_{j \in J^{i}} c^j \cdot \left( \xi(B^{j}) - \frac{p^{i}}{\left| J^{i} \right|} \right) .
\end{align*}
We bound these three terms in turn. The first term is controlled by $r \cdot \mathrm{Lip}(\phi_\mu) \leq r\sqrt{2} L_K$. The second term is bounded by $\xi(B^{c}_{r}) \cdot \|\phi_\mu - \lambda_{1}(\mu)\|_{\mathscr{C}(X)}$, where $\|\phi_\mu - \lambda_{1}(\mu)\|_{\mathscr{C}(X)} \leq \sup_{g\in K} \log N(g)^2 \leq \log L_K$. For the third term, $\max_j|c^j|$ is bounded by $\mathrm{Lip}(\phi_\mu) \leq \sqrt{2}L_K$. The lemma follows.
\end{proof}

% ----------------------------------------------------------------------
\section{Random walk estimates}\label{sec:rw}

By Lemma~\ref{lemma:estimate2}, the convergence of $\int_X \phi_\mu\,d\nu_n$ to $\lambda_1(\mu)$ is controlled by two quantities: the mass $\nu_n(B_r^c)$ that the measure places outside the $r$-neighborhoods of the subspaces $\mathbb{P}(W^j)$ (the accumulation term), and the imbalance $|\nu_n(B^j) - \nu_n(B^{j'})|$ between neighboring neighborhoods in the same $F_\mu$-orbit (the equidistribution term). In this section we derive polynomial bounds for both, using a Berry--Esseen estimate for the logarithmic slopes of the conformal blocks.

We work with the coarsest virtually $\Gamma_\mu$-conformal decomposition $\mathbb{R}^d = \oplus_{j\in J} W^j$ from Proposition~\ref{prop:struc_assump} and fix a basis accordingly. The maps $a(g)$ and $q(g)$ are well-defined with respect to this basis (Definition~\ref{def:decomp}). We write $\mathbb{P}=\mu^{\mathbb{N}^*}$, $\mathbb{E} = \int \cdot\, d\mathbb{P}$, $a_n = a(g_n)$, and $q_n = q(g_n)$. Recall from Definition~\ref{def:decomp} that the norms on each $W^j$ may be adjusted by a positive scalar.

\subsection{Berry-Esseen bounds for the slopes}

We utilize Schneider's theorem for exponentially $\varphi$-mixing sequences to control the random walk behavior of the logarithmic slopes. Let $\Psi(x) = \int_{-\infty}^x \frac{1}{\sqrt{2\pi}} e^{-z^2/2} dz$ denote the standard Gaussian cumulative distribution function.

\begin{thm}[Schneider~\cite{Sch81}]\label{lemma:sch0}
Let $(Z_{m})_{m \in \mathbb{N}}$ be a sequence of random variables satisfying
\begin{enumerate}[label=(\roman*)]
    \item $\mathbb{E}[Z_{m}]=0$ for every $m \geq 0$ and $\mathrm{sup}_{m \geq 0} \mathbb{E}\left[ \left| Z_{m} \right|^{3} \right] < \infty$,
    \item $\liminf_{ n \to \infty } \frac{w_{n}^{2}}{n}>0$ where $w^{2}_{n} =\mathbb{E}\left[ \left| \sum_{m=1}^{n} Z_{m} \right|^{2} \right]$, and
    \item there exist constants $C_{\text{mix}}>0$ and $c>0$ such that
\end{enumerate}
$$
\left| \mathbb{P}[B\mid A] - \mathbb{P}[B] \right| \leq C_{\text{mix }} e^{-c \jmath},
$$
for any $A \in \sigma(Z_{1}, \dots, Z_{l})$, $B \in \sigma(Z_{l+\jmath+1},\dots, Z_{l+\jmath+l'})$ and any $\jmath,l,l' \in \mathbb{N}$.

Then for any $t \in \mathbb{R}$ and $n \geq 2$,
$$
\left| \mathbb{P}\left[\frac{1}{ w_{n}} \sum_{m=1}^{n}Z_{m} \leq t \right] - \Psi(t) \right| \leq C_{\text{mix }} \frac{\log n}{\sqrt{ n }}.
$$
Consequently, for any $c_{1}<c_{2}$ and any $n \ge 2$,
$$
\mathbb{P}\left[ c_{1} \leq \frac{1}{w_{n}} \sum_{m=1}^{n} Z_{m} \leq c_{2} \right] \leq C_{\text{mix }} \left( \frac{\log n}{\sqrt{ n }} + {c_{2}-c_{1}} \right)
$$
\end{thm}

The exact zero-mean requirement in condition (i) of Theorem \ref{lemma:sch0} is inconvenient for sequences arising in our setting. While it does not strictly relax the assumptions due to the necessary martingale difference condition, we adapt the theorem for convenience to handle sequences whose expectations decay exponentially:

\begin{thm}\label{lemma:sch}
Let $(Z_{m})_{m \in \mathbb{N}}$ be a sequence of random variables satisfying
\begin{enumerate}
    \item[(i')] $\left|\mathbb{E} Z_{m} \right| \leq C_{0}\rho ^{m}$ for some constants $C_{0}>0$ and $\rho \in (0,1)$, for every $m \geq 0$, with $\mathrm{sup}_{m \in \mathbb{N}} \left| Z_{m} \right| < \infty$,
    \item[(ii')] $(Z_{m}-\mathbb{E}Z_{m})_{m\in \mathbb{N}}$ is a martingale difference sequence and $\liminf_{ n \to \infty } \frac{w_{n}^{2}}{n}>0$ where $w^{2}_{n} =\sum_{m=1}^{n} \mathbb{E}\left[Z_{m}^{2} \right]$, and
    \item[(iii)] there exist constants $C_{\text{mix }}>0$ and $c>0$ such that $\left| \mathbb{P}[B\mid A] - \mathbb{P}[B] \right| \leq C_{\text{mix }} e^{-c \jmath}$,
\end{enumerate}
for any $A \in \sigma(Z_{1}, \dots, Z_{l})$, $B \in \sigma(Z_{l+\jmath+1},\dots, Z_{l+\jmath+l'})$ and any $\jmath,l,l' \in \mathbb{N}$.

Then for any $t \in \mathbb{R}$ and $n \geq 2$, there exists a constant $C> 0$ such that
$$
\left| \mathbb{P}\left[\frac{1}{ w_{n}} \sum_{m=1}^{n}Z_{m} \leq t \right] - \Psi(t) \right| \leq C \frac{\log n}{\sqrt{ n }}.
$$
Consequently, for any $c_{1}<c_{2}$ and any $n \ge 2$,
$$
\mathbb{P}\left[ c_{1} \leq \frac{1}{w_{n}} \sum_{m=1}^{n} Z_{m} \leq c_{2} \right] \leq C \left( \frac{\log n}{\sqrt{ n }} + {c_{2}-c_{1}} \right)
$$
\end{thm}

\begin{proof}
We apply Theorem \ref{lemma:sch0} to the centered sequence $\tilde{Z}_m = Z_{m}-\mathbb{E}Z_{m}$. Condition (i) is satisfied because the variables $\tilde{Z}_m$ are explicitly zero-mean, and the uniform bound on $|Z_m|$ implies their third absolute moments are uniformly bounded. For condition (ii), since $(\tilde{Z}_m)_{m \in \mathbb{N}}$ is a martingale difference sequence, the sum of variances matches the variance of the sum:
$$
u^{2}_{n} := \mathbb{E} \left[  \left| \sum_{m=1}^{n} \tilde{Z}_m \right|^{2} \right] = \sum_{m=1}^{n} \mathbb{E}\left[  | \tilde{Z}_m |^{2}  \right] = w^{2}_{n} - \sum_{m=1}^{n} (\mathbb{E}Z_{m})^{2} .
$$
Because condition (i') enforces that expectations decay exponentially, the infinite series $\sum_{m=1}^{\infty} (\mathbb{E}Z_{m})^{2}$ converges to a finite constant. Since $\liminf \frac{w_n^2}{n} > 0$, the sum $w_n^2$ grows linearly without bound. Therefore, the asymptotic lower bounds are identically non-zero:
$$
\liminf_{ n \to \infty } \frac{u^{2}_{n}}{n} = \liminf_{ n \to \infty } \frac{w^{2}_{n}}{n} > 0.
$$
Condition (iii) on exponential $\varphi$-mixing is unaffected by the centering, since deterministic shifts do not alter the $\sigma$-algebras generated by the sequence. Thus, Theorem \ref{lemma:sch0} applies directly to the sum $Y_n = \frac{1}{u_n} \sum_{m=1}^n \tilde{Z}_m$, yielding a Berry-Esseen bound with the mixing constant $C_{\mathrm{mix}}>0$ such that for all $x \in \mathbb{R}$ and $n \ge 2$:
$$
\left| \mathbb{P}(Y_n \leq x) - \Psi(x) \right| \leq C_{\mathrm{mix}} \frac{\log n}{\sqrt{n}}.
$$

To transfer this bound back to the original uncentered sum, define the scaling factor $\kappa_n = \frac{u_n}{w_n} \in (0,1]$ and the shift $\beta_n = \frac{1}{w_n} \sum_{m=1}^n \mathbb{E}Z_m$. The normalized original sum is then given exactly by $\frac{1}{w_n} \sum_{m=1}^n Z_m = \kappa_n Y_n + \beta_n$. We evaluate the target probability and bound the deviation using the triangle inequality:
\begin{align*}
\left| \mathbb{P}\!\left(Y_n \leq \tfrac{t-\beta_n}{\kappa_n}\right) - \Psi(t) \right|
&\leq \left| \mathbb{P}\!\left(Y_n \leq \tfrac{t-\beta_n}{\kappa_n}\right) - \Psi\!\left(\tfrac{t-\beta_n}{\kappa_n}\right) \right|
+ \left| \Psi\!\left(\tfrac{t-\beta_n}{\kappa_n}\right) - \Psi(t) \right|.
\end{align*}
The first term is bounded by $C_{\mathrm{mix}} \frac{\log n}{\sqrt{n}}$. 
For the second term, we have $\left| \Psi\!\left(\tfrac{t-\beta_n}{\kappa_n}\right) - \Psi\left(\tfrac{t}{\kappa_n}\right) \right| \leq \frac{|\beta_n|}{\sqrt{2\pi}\kappa_n}$ and:
$$
\left| \Psi(\tfrac{t}{\kappa_n}) - \Psi(t) \right| \leq \frac{1}{\sqrt{2\pi}} \left| \tfrac{t}{\kappa_n} - t \right| e^{-t^2/2} \leq 
\frac{1-\kappa_n}{\sqrt{2\pi}\kappa_n}.
$$

By condition (ii'), $w_n^2 \ge c_w n$ for some uniform constant $c_w > 0$ and all $n \ge 1$. From $w_n^2 \ge c_w n$, we get the bound on the scaling ratio:
$$
1 \ge \kappa_n = \left( 1 - \frac{1}{w^{2}_{n}} \sum_{m=1}^{n} (\mathbb{E}Z_{m})^{2} \right)^{1/2} \geq 1 - \frac{1}{c_w n} \frac{C_{0}^{2}}{1-\rho ^{2}}.
$$
This establishes $1 \ge \kappa_n \ge 1 - \frac{C_2}{n}$ for a constant $C_2=C_2(C_0,\rho,c_w) > 0$. Similarly, the shift satisfies:
$$
|\beta_n| \leq \frac{1}{\sqrt{c_w n}} \frac{C_{0}}{1-\rho} \leq \frac{C_3}{\sqrt{n}}.
$$
For $n \ge 2$ large enough that $\kappa_n \ge \frac{1}{2}$, we obtain:
$$
\left| \Psi(t_n) - \Psi(t) \right| \leq \frac{2}{\sqrt{2\pi}} \left( \frac{C_2}{n} + \frac{C_3}{\sqrt{n}} \right) \leq \frac{C_4}{\sqrt{n}}.
$$
Combining both bounds gives the claimed inequality with $C = C_{\mathrm{mix}} + C_4$ for all $n$ large enough.
\end{proof}

We now define the random variables to which Theorem~\ref{lemma:sch} will be applied. 
Recall that $a_m=a(g_m)$ and $J$ indexes the components of a coarsest virtually $\Gamma_\mu$-conformal decomposition of $\mathbb{R}^d$. 
For $j, j' \in J$, define the \emph{logarithmic slope} at step $m$ by
$$
\alpha _{m}^{j',j} = \log \frac{a^{\zeta _{m-1}(j')}_{m}}{a^{\zeta _{m-1}(j)}_{m}},
$$
where $\zeta _{m} = \pi _{m}\dots \pi_{1} \in S_{J}$ and $\pi_m$ is defined by $g_{m}W^{j}=W^{\pi _{m}(j)}$. To verify the hypotheses of Theorem~\ref{lemma:sch} for these slopes, we establish in turn: exponential decay of the mean (Lemma~\ref{lemma:mean}), centering within each $F_\mu$-orbit (Lemma~\ref{lemma:centering}), non-degeneracy of the variance (Lemma~\ref{lemma:variance}), and the mixing condition (Lemma~\ref{lemma:mix}). 

\begin{lemma} \label{lemma:mean}
Let $\mu$ be a compactly supported semisimple probability measure on $G$ with $\lambda_1(\mu) = \lambda_d(\mu)$, and let $\oplus_{j\in J}W^j$ be a coarsest virtually $\Gamma_\mu$-conformal decomposition. Assume $\mu$ is aperiodic in $F_\mu$. Then there exist constants $C_0>0$ and $\rho\in (0,1)$, such that for any $j$ and $j'$ in different $F_{\mu}$-orbits,
$$
\left| \mathbb{E}[\alpha _{n}^{j',j}] \right| \leq C_0 \rho ^{n-1}.
$$
\end{lemma}
\begin{proof}
By expanding the expectation conditionally on the finite group permutations, we have:
$$
\mathbb{E}[ \alpha _{n}^{j',j} ] = \sum_{s \in F _{\mu}} \mathbb{E} \left[\log{\frac{a^{s(j')}_{n}}{a^{s(j)}_{n}}} \right] \mathbb{P}(\zeta _{n-1}=s).
$$
Let $\zeta _{\infty}$ be a uniformly distributed random variable on $F_\mu$. By Lemma \ref{lemma:const}, the steady-state drift perfectly cancels because $\lambda_1(\mu)$ evaluates identically on all $F_\mu$-orbits:
$$
0 = \sum_{s \in  F _{\mu}} \mathbb{E} \left[  \log{ \frac{a^{s(j')}_{n}}{a^{s(j)}_{n}}} \right] \mathbb{P}(\zeta _{\infty} = s).
$$
Subtracting these two equalities gives us the bounded difference:
$$
\left| \mathbb{E}[\alpha _{n}^{j',j}] \right| \leq \sum_{s \in F _{\mu}} \left| \mathbb{E} \left[  \log{\frac{a^{s(j')}_{n}}{a^{s(j)}_{n}}} \right] \right| \left| \mathbb{P}(\zeta _{n-1} = s) - \mathbb{P}(\zeta _{\infty} = s) \right|.
$$
By applying the exponential convergence of the aperiodic Markov chain $(\zeta_n)_n$ on the finite group $F_\mu$ (Lemma \ref{lemma:markov}), we bound the variation distance strictly by $\rho ^{n-1}$. We bound the expected scaling factors using compactness:
$$
\left| \mathbb{E}\left[ \log{\frac{a^{s(j')}_{n}}{a^{s(j)}_{n}}} \right] \right| \leq \mathbb{E} \left[ \log{\max\left(  \frac{a^{s(j')}_{n}}{a^{s(j)}_{n}}, \frac{a^{s(j)}_{n}}{a^{s(j')}_{n}} \right)} \right] \leq \mathbb{E}\left[ \log{ N(g_n)^2} \right].
$$
Since $\log N(g)^2 \le \log L_K$ and $|F_\mu| \le d!$, we obtain
$$
\left| \mathbb{E}[\alpha _{n}^{j',j}] \right| \leq (d!)(\log L_K) \rho^{n-1}.
$$
Setting $C_0 = (d!)(\log L_K)$ completes the proof.
\end{proof}

When $j$ and $j'$ lie in the same $F_\mu$-orbit, the expected value of $\alpha_m^{j',j}$ need not be small, but it can be made exactly zero by a suitable rescaling of the conformal norms.

\begin{lemma} \label{lemma:centering}
Let $\mu$ be a compactly supported semisimple probability measure on $G$ with $\lambda_1(\mu) = \lambda_d(\mu)$, and let $\oplus_{j\in J}W^j$ be a coarsest virtually $\Gamma_\mu$-conformal decomposition. After a suitable change of the basis conforming to the conformal structures $[\mathcal{C}^j]$ on $W^j$, we have:
$$
\int_G \log \frac{a^{j}(g)}{a^{j'}(g)} \, d \mu(g) = 0,
$$
for any $j, j' \in J$ in the same $F_\mu$-orbit.
\end{lemma}
\begin{proof}
To ensure the integrals vanish, we independently rescale the basis vectors of each subspace. Let $J^i \subset J$ be an $F_\mu$-orbit. Because the action of $\Gamma_\mu$ permutes the subspaces $W^j$ for $j \in J^i$ transitively, the permutations $s \in F_\mu$ act transitively on $J^i$.

For every $j \in J^i$, we apply a scaling transformation $T = \oplus_{j \in J^i} t^j \mathrm{Id}_{W^{j}}$ with strictly positive scalar factors $t^j > 0$. Under this change of coordinates, the action of an element $g \in \Gamma_\mu$ restricted to the component $W^j$ becomes:
$$
g|_{W^j} = T|_{W^{s(j)}} \cdot a^{s(j)}(g) q^s|_{W^j} \cdot T^{-1}|_{W^j} = \left( a^{s(j)}(g) \frac{t^{s(j)}}{t^j} \right) q^s|_{W^j}.
$$
In this newly rescaled basis, the block scalar factor becomes $\tilde{a}^j(g) = a^j(g) \frac{t^j}{t^{s^{-1}(j)}}$. Our algebraic goal is to choose the scalars $t^j$ such that the drift,
$$
C^j := \int_G \log \tilde{a}^j(g) \, d\mu(g),
$$
is strictly independent of $j$ for all $j \in J^i$. If $C^j \equiv C$ operates as a uniform constant, then for any arbitrary elements $j, j' \in J^i$, the difference of their identical integrals naturally vanishes.

Let $c^j = \int_G \log a^j(g) \, d\mu(g)$ represent the unscaled drift, and define $z^j = \log t^j$. The expected logarithm of the new scalar factor expands to:
$$
\int_G \log \tilde{a}^j(g) \, d\mu(g) = c^j + z^j - \sum_{s \in F_\mu} \mu(\{s\}) z^{s^{-1}(j)}.
$$
We rewrite the resulting summation using the transition probability matrix $P$ of the random walk mapped on the finite state space $J^i$, constructed fundamentally by $P_{j,k} = \mu(\{s \in F_\mu : s^{-1}(j) = k\})$. Because the generating set $F_\mu$ functions transitively, $P$ is an irreducible, doubly stochastic matrix. We thereby seek a column vector $z = (z^j)_{j \in J^i}$ and a uniform constant $C$ such that:
$$
c + z - Pz = C \mathbf{1},
$$
where $\mathbf{1}$ denotes the vector populated completely by ones. Rearranging this relation yields Poisson's equation evaluated for the Markov chain:
$$
(\mathrm{Id} - P)z = C \mathbf{1} - c.
$$
The operator $\mathrm{Id} - P$ is invertible on the subspace orthogonal to the stationary distribution. Since $P$ is doubly stochastic, its unique stationary distribution is the uniform measure. A solution $z$ therefore exists if and only if the right-hand side has zero mean:
$$
\sum_{j \in J^i} (c^j - C) = 0 \implies C = \frac{1}{|J^i|} \sum_{j \in J^i} c^j.
$$
With this specific parameter choice, the vector $C \mathbf{1} -c$ lies in the image of $\mathrm{Id} - P$. We take $z$ to be any solution of this linear system and set $t^j = \exp(z^j)$. Repeating this coordinate change for each orbit $J^i$ completes the proof.
\end{proof}

\begin{rmk}
The block-conformal change of coordinates $T = \oplus_{j\in J} t^j\mathrm{Id}_{W^j}$ in Lemma~\ref{lemma:centering} replaces the norms on each $W^j$ by equivalent ones. Since the Lyapunov exponents are independent of the choice of basis on $\mathbb{R}^d$ and the change of coordinates is Lipschitz with respect to the probability measures in the Wasserstein distance, the log-H\"{o}lder bound of Theorem~\ref{thmA} is unaffected by this change.
\end{rmk}

The non-degeneracy of the variance is the key quantitative input that prevents the random walk from being deterministic; it is what ensures the Berry--Esseen bound gives a useful decay rate.
\begin{lemma} \label{lemma:variance}
For a coarsest virtually $\Gamma_\mu$-conformal decomposition $\oplus_{j\in J} W^j$, the variance 
$$
w^{j',j}:= \frac{1}{|F_\mu|}\sum_{s\in F_\mu} \mathbb{E}\left[ \log^2 \frac{a^{s(j')}}{a^{s(j)}} \right]
$$
does not vanish for any distinct $j,j'\in J$. 
\end{lemma}
\begin{proof}
Suppose for contradiction that there exist distinct $j,j'\in J$ with $w^{j',j}=0$, that is, $a^{s(j')} \equiv a^{s(j)}$ for any $s \in F _{\mu}$. The equality $a^{i'} \equiv a^{i}$ defines an equivalence relation $i' \sim i$ on $J$, and any $s\in F _{\mu}$ maps an equivalence class to another one. Let $U^i$ be the direct sum $\oplus _{i'\in i} W^{i'}$ of an equivalence class $i$. We find that every $U^{i}$ is a virtually $\Gamma _{\mu}$-invariant subspace and $\Gamma _{\mu}$ is conformal on $U^{i}$. This means that $\oplus _{i \in I} U^{i}$ is a virtually $\Gamma _{\mu}$-conformal decomposition that is coarser than $\oplus _{j\in J}W^{j}$. At least one subspace $U^{i}$ is strictly larger than $W^{i'}$ because there is an equivalent class that has at least $2$ elements, $j$ and $j'$. This cannot happen since $\oplus _{j\in J}W^{j}$ is already a coarsest one. 
\end{proof}

With the mean, centering, and variance established, we now verify all three conditions of Theorem~\ref{lemma:sch} together.
\begin{lemma}\label{lemma:mix}
Let $\mu$ be a compactly supported semisimple probability measure on $G$ with $\lambda_1(\mu) = \lambda_d(\mu)$. For every $j,j' \in J$, the sequence of logarithmic slopes $\alpha _{m}^{j,j'}$ satisfies all conditions required by Theorem \ref{lemma:sch}.
\end{lemma}
\begin{proof}
Condition~(i'). By Lemma~\ref{lemma:centering}, $\mathbb{E}[\alpha_n^{j',j}] = 0$ when $j$ and $j'$ belong to the same $F_\mu$-orbit. When they belong to different orbits, $|\mathbb{E}[\alpha_n^{j',j}]| \leq C_0\rho^{n-1}$ by Lemma~\ref{lemma:mean}. The uniform bound on $|\alpha_m^{j',j}|$ follows from compactness of $K$.

Condition~(ii'). Expanding the expectation using the indicator of the permutation state and the independence of $\zeta_{m-1}$ from $a_m$ (which has the same distribution as $a_1$):
\begin{align*}
w_{n}^{2}  &:= \sum_{m=1}^{n} \mathbb{E}\!\left[ \left| \alpha _{m}^{j',j} \right|^{2} \right]
= \sum_{m=1}^{n} \mathbb{E}\!\left[ \sum_{s \in F_\mu} \mathbf{1}_{\zeta _{m-1}=s} \log ^{2} \frac{a_{m}^{s(j')}}{a_{m}^{s(j)}} \right]
= \sum_{m=1}^{n} \sum_{s \in F_\mu} \mathbb{E}\!\left[ \log ^{2} \frac{a_{1}^{s(j')}}{a_{1}^{s(j)}} \right] \mathbb{P}[\zeta _{m-1} =s].
\end{align*}
Set $f(s) = \mathbb{E}\!\left[ \log ^{2} \frac{a_{1}^{s(j')}}{a_{1}^{s(j)}} \right] \geq 0$. Since the Markov chain $(\zeta_m)_m$ converges to the uniform distribution $\zeta_\infty$ on $F_\mu$ (Lemma~\ref{lemma:markov}):
$$
\frac{1}{n} w_{n}^{2} = \frac{1}{n} \sum_{m=1}^{n} \mathbb{E}[f(\zeta _{m-1})] \longrightarrow \mathbb{E}[f(\zeta _{\infty})] = \frac{1}{\left| F_\mu \right|}\sum_{s \in F_\mu} f(s) > 0,
$$
where positivity follows from Lemma~\ref{lemma:variance}. This gives $\lim_{n\to\infty} w_n^2/n > 0$.

The martingale difference property of $(\alpha_m^{j',j} - \mathbb{E}[\alpha_m^{j',j}])_{m\geq 1}$ with respect to the natural filtration $\mathcal{F}_m = \sigma(g_1,\ldots,g_m)$ is immediate from the i.i.d.\ structure of $(g_m)$.

Condition~(iii). Let $\mathcal{F}_l = \sigma(g_1, \dots, g_l)$ be the natural filtration. Clearly, the past slopes generate a sub-algebra $\sigma(Z_1, \dots, Z_l) \subset \mathcal{F}_l$. 

Any future event $B \in \sigma(Z_{l+\jmath+1}, \dots)$ is completely determined by the intermediate permutation state $\zeta_{l+\jmath}$ and the future group increments $(g_m)_{m > l+\jmath}$. Since these future increments are i.i.d.\ and independent of $\mathcal{F}_{l+\jmath}$, the conditional probability of $B$ given the past is determined entirely by the finite state at time $l+\jmath$. Specifically, there exists a deterministic function $f: F_\mu \to [0,1]$ such that $\mathbb{P}(B \mid \mathcal{F}_{l+\jmath}) = f(\zeta_{l+\jmath})$.

For any past event $A \in \sigma(Z_1, \dots, Z_l)$ with $\mathbb{P}(A)>0$, the state $\zeta_l$ is $\mathcal{F}_l$-measurable. The transition from $\zeta_l$ to $\zeta_{l+\jmath}$ is driven by $\jmath$ independent permutations $\pi_{l+\jmath} \dots \pi_{l+1}$. By Lemma \ref{lemma:markov} applied to the aperiodic random walk on the finite group $F_\mu$, the distribution of $\zeta_{l+\jmath}$ converges to the uniform measure exponentially fast, uniformly with respect to any initial distribution: there exists constants $C_1 > 0$ and $c>0$ that do not depend on $A$ or $B$, such that:
$$
\max_{s \in F_\mu} \left| \mathbb{P}(\zeta_{l+\jmath} = s \mid A) - \mathbb{P}(\zeta_{l+\jmath} = s) \right| \leq 2 C_1 e^{-c\jmath}.
$$
Integrating $f$ against these distributions yields the $\varphi$-mixing bound:
$$
\left| \mathbb{P}(B \mid A) - \mathbb{P}(B) \right| = \left| \sum_{s \in F_\mu} f(s) \left( \mathbb{P}(\zeta_{l+\jmath} = s \mid A) - \mathbb{P}(\zeta_{l+\jmath} = s) \right) \right| \leq 2 |F_\mu| C_1 e^{-c\jmath},
$$
with $C_{\mathrm{mix}} = 2(d!)C_1$.
\end{proof}

\subsection{Proof of the accumulation estimate}

\begin{lemma}\label{lemma:accum}
Let $x = \mathbb{R}(v^{j})_{j \in J} \in X = \mathbb{P} (\oplus_{j\in J} W^j)$. We have
$$
\mu ^{*n}*\delta_x ( B_{r}^{c} ) \leq \sum_{j \neq j' \in J}\mathbb{P}\left( (g_n)_n: -R \leq \log \frac{\|v^{j'}\|}{\|v^{j}\|} + \sum_{m=1}^{n} \alpha _{m}^{j',j} \leq R \right),
$$
where $R=-\log{ \frac{r}{\sqrt{ d }}}$. We interpret the event as empty if the denominator $\| v^j \| =0$.
\end{lemma}
\begin{proof}
A point $x \in X$ belongs to $B_{r}^{c}$ if and only if for every $j \in J$, 
$$
d_X(x, \mathbb{P}(W^{j})) = \inf_{w \in W^{j} \setminus \left\{ 0 \right\}} |\sin \angle(v,w)| \geq r,
$$
where $x=\mathbb{R}v$. Write $v = \sum_{j \in J} v^{j}$ according to the decomposition. For $w \in W^j \setminus \{0\}$, 
$$
d_X(x, \mathbb{P}(W^{j})) = \sqrt{ 1 - \frac{\|v^{j}\|^{2}}{ \sum_{j' \in J} \|v^{j'}\|^{2}} }.
$$

The point $x$ belongs to $B_{r}^{c}$ if and only if for every $j \in J$, 
$$
\frac{\sum_{j' \in J \setminus \left\{ j \right\}} \|v^{j'}\|^{2}}{\|v^{j}\|^{2}} \geq \frac{1}{1-r^{2}} - 1.
$$
If so, there exists some $j' \in J \setminus \left\{ j \right\}$ with $v^{j'} \neq 0$, such that $\frac{\|v^{j'}\|}{\|v^{j}\|} \geq \frac{r}{\sqrt{ d}}$. It follows that
$$
B_{r}^{c} \subset \cap _{j \in J}\cup _{j' \in J\setminus \left\{ j \right\}} \left\{ x: \frac{\|v^{j'}\|}{\|v^{j}\|} \geq e^{-R} \right\} ,
$$
where $R = -\log \frac{r}{\sqrt{ d }}$. We claim that
\begin{equation} \label{inclusion}
\cap _{j \in J}\cup _{j' \in J\setminus \left\{ j \right\}} \left\{ x = \R v: v^{j'} \neq 0 \text{ and } \frac{\|v^{j'}\|}{\|v^{j}\|} \geq e^{-R} \right\} \subset \cup _{j \neq j' \in J} \left\{ x = \R v :  -R \leq \log \frac{\|v^{j'}\|}{\|v^{j}\|} \leq R \right\},
\end{equation}
which we verify at the end of the proof. Accepting the inclusion for now, we have
$$
\mu ^{*n}*\delta_x (B_{r}^{c}) = \mathbb{P}(g_{n}\dots g_{1}x \in B_{r}^{c}) \leq \sum_{j \neq j' \in J} \mathbb{P}\left( -R \leq \log \frac{\|v_{n}^{j'}\|}{\|v_{n}^{j}\| } \leq R \right).
$$

Write $v = \sum_{j \in J} v^j$ with $v^j \in W^j$, and let $v_n^j$ denote the $W^j$-component of $g_n \cdots g_1 v$. At each step $m$, the matrix $g_m$ maps $W^{\zeta_{m-1}(j)}$ to $W^{\zeta_m(j)}$ by a conformal transformation with isometric part $q_m^{\zeta_{m-1}(j)} \in \mathrm{O}(\dim W^j)$ and scalar factor $a_m^{\zeta_{m-1}(j)}$ (Definition~\ref{def:decomp}). Therefore the $W^{\zeta_m(j)}$-component of the walk satisfies, at each step,
\begin{equation}\label{eq:step}
\|v_m^{\zeta_m(j)}\| = a_m^{\zeta_{m-1}(j)} \cdot \|v_{m-1}^{\zeta_{m-1}(j)}\|.
\end{equation}
Telescoping \eqref{eq:step} over $m = 1, \ldots, n$ and taking logarithms:
$$
\log \|v_n^{\zeta_n(j)}\| = \log \|v^j\| + \sum_{m=1}^n \log a_m^{\zeta_{m-1}(j)}.
$$
Subtracting this identity for $j$ from the one for $j'$ and recalling $\alpha_m^{j',j} = \log a_m^{\zeta_{m-1}(j')} - \log a_m^{\zeta_{m-1}(j)}$:
\begin{equation}\label{eq:coord_exp}
\log \frac{\|v_n^{\zeta_n(j')}\|}{\|v_n^{\zeta_n(j)}\|} = \log \frac{\|v^{j'}\|}{\|v^{j}\|} + \sum_{m=1}^{n} \alpha_m^{j',j}.
\end{equation}
Since $\zeta_n \in F_\mu$ is a bijection on $J$, writing $\tilde j = \zeta_n(j)$ and $\tilde j' = \zeta_n(j')$ the event $\{-R \leq \log \|v_n^{\tilde j'}\|/\|v_n^{\tilde j}\| \leq R\}$ ranges over all pairs $\tilde j \neq \tilde j'$ just as $(j,j')$ do. Substituting \eqref{eq:coord_exp} into the probability bound above yields the statement.

To verify the inclusion~\eqref{inclusion}, let $x=\mathbb{R}v$ belong to the left-hand side and let $j \in J$ be an index maximizing $\|v^{j}\| > 0$. Then for $j' \in J \setminus \left\{ j \right\}$ given by the left-hand side, $-R \leq \log \frac{\|v^{j'}\|}{\|v^{j}\|} \leq 0 \leq R$.
\end{proof}

%\subsection{Proof of the Accumulation Estimate}

Combining Lemma~\ref{lemma:accum} with the Berry--Esseen bound from Theorem~\ref{lemma:sch} gives the following polynomial bound on the mass placed outside the neighborhoods of the conformal subspaces.

\begin{prop}[Accumulation]\label{prop:accum}
There exist a universal constant $\gamma>0$ and a constant $C(K,\mu)>0$, such that for any $x \in X$, any $r\in\left( 0, \frac{1}{2} \right)$ and any $n \in \mathbb{N}^{*}$,
$$
\mu ^{*n}*\delta_x ( B_{r}^{c} ) \leq \frac{C}{n^{\gamma}}(1 - \log r).
$$
Hence for any probability measure $\nu'$ on $X$, any $r \in (0, \frac{1}{2})$ and any $n\in \mathbb{N}^*$,
$$
\nu _{n}(B_{r}^{c}) \leq \frac{C}{n^{\gamma}} (1-\log r).
$$
\end{prop}
\begin{proof}
By Lemma~\ref{lemma:mix}, the sequence $(\alpha_m^{j,j'})_m$ satisfies all conditions of Theorem~\ref{lemma:sch}. In particular, there exists a Berry--Esseen bound: for each pair $j \neq j'$ and a constant $C_1 > 0$,
$$
\mathbb{P}\left( b: -R \leq \log \frac{\|v^{j'}\|}{\|v^{j}\|} + \sum_{m=1}^{n} \alpha _{m}^{j,j'} \leq R \right)
\leq C_1 \left( \frac{\log n}{\sqrt{ n }} + \frac{2R}{w\sqrt{ n }} \right),
$$
where $w = \min_{j \neq j'} w^{j',j} > 0$. By Lemma~\ref{lemma:accum}, summing over all pairs and writing $R = -\log r + \frac{1}{2}\log d$,
$$
\mu ^{*n}(g: gx \in B^{c}_{r}) \leq C_2 \frac{\log n - \log r}{\sqrt{ n }} \leq \frac{C}{n^{\gamma}}(1 - \log r),
$$
where $\gamma \in (0, 1/2]$ is a universal constant chosen so that the logarithmic factor is absorbed.
\end{proof}

\subsection{Proof of the equidistribution estimate}

The accumulation estimate controls how much mass lies far from all the subspaces $\mathbb{P}(W^j)$. The equidistribution estimate controls the imbalance between the mass near different subspaces in the same $F_\mu$-orbit; this is where the permutation structure of $F_\mu$ plays a role. We work with the sharper neighborhoods
$$
U^{j}_{R} := \bigcap_{j' \in J\setminus \left\{ j \right\}  } \left\{ \mathbb{R} v \in X: \log{\frac{\|v^{j}\|}{\|v^{j'}\|}} > R \right\},
$$
where $\log \frac{\| v^j \|}{0} > R$ is interpreted as true when $\| v^j \| \neq 0$ and false otherwise.

\begin{prop}[Equidistribution]\label{prop:equid}
There exist a universal constant $\gamma>0$ and a constant $C(K,\mu)>0$, such that if $R = O(\log n)$, then for any $x \in X$, $n \in \mathbb{N}^{*}$, $i \in I$ and $j,j'\in J^{i}$,
$$
\left| \mu ^{*n}*\delta _{x}(U^{j'}_{R}) - \mu ^{*n}*\delta _{x}(U^{j}_{R}) \right| \leq  \frac{C}{n^{\gamma}} .
$$
Hence for any probability measure $\nu'$ on $X$ and any $n \in \mathbb{N}^{*}$,
$$
\sum_{i\in I} \sum_{j,j' \in J^i} \left| \nu _{n}(B^{j}) - \nu _{n}(B^{j'}) \right| \leq  \frac{C}{n^{\gamma}}.
$$
\end{prop}
\begin{proof}
Denote $\omega_n^j = \mu^{*n}*\delta_x(U^j_R)$, $\omega_n = (\omega_n^j)_{j\in J^i} \in \mathbb{R}^{J^i}$, and $\omega_n^{j,j'} = |\omega_n^j - \omega_n^{j'}|$.

The key step is a one-step recursion that converts the equidistribution question into a problem about the spectral gap of the permutation walk.

\begin{lemma}[Recursion]\label{lemma:recursion}
There exist $C>0$ and an irreducible stochastic matrix $Q = Q(\mu)$ with entries $Q^{j,j_0} = \mathbb{P}(\pi_1(j_0) = j)$, such that
$$
\omega_{n} \leq Q^{n}\omega_{1} + C \sum_{k=1}^{n-1}Q^{n-k-1} \frac{R + \log k}{\sqrt{k}}\mathbf{1}.
$$
Here the inequality is entry-wise and $\mathbf{1}$ denotes $(1,1,\dots,1) \in \mathbb{R}^{J^i}$.
\end{lemma}
\begin{proof}
Fix $j \in J^i$. Since $(g_m)_{m\geq 1}$ are i.i.d., conditioning on the permutation $\pi_{n+1}$ of the last step:
\begin{align}
\mu^{*(n+1)}*\delta_x(U^j_R)
&= \mathbb{P}\!\left(\log\frac{\|v^j_{n+1}\|}{\|v^{j'}_{n+1}\|} > R,\;\forall j'\in J\setminus\{j\}\right) \nonumber\\
&= \sum_{s\in F_\mu} \mathbb{P}\!\left(\log\frac{\|v^{s^{-1}(j)}_n\|}{\|v^{s^{-1}(j')}_n\|} + \log\frac{a^{s^{-1}(j)}_{n+1}}{a^{s^{-1}(j')}_{n+1}} > R,\right.\nonumber\\
&\qquad\qquad\left.\forall j'\in J\setminus\{j\} \;\Big|\; \pi_{n+1}=s\right)\mathbb{P}(\pi_{n+1}=s). \label{eqn:1}
\end{align}
We estimate each conditional factor. For any fixed $s\in F_\mu$ and $j'\in J$, independence of $a_{n+1}$ from $g_1,\ldots,g_n$ gives:
\begin{align}
&\left|\mathbb{P}\!\left(\log\frac{\|v^{s^{-1}(j)}_n\|}{\|v^{s^{-1}(j')}_n\|} + \log\frac{a^{s^{-1}(j)}_{n+1}}{a^{s^{-1}(j')}_{n+1}} > R\;\Big|\;\pi_{n+1}=s\right)\right.\nonumber\\
&\quad - \left.\mathbb{P}\!\left(\log\frac{\|v^{s^{-1}(j)}_n\|}{\|v^{s^{-1}(j')}_n\|} > R\right)\right| \nonumber\\
\leq\;& \mathbb{P}\!\left(\log\frac{\|v^{s^{-1}(j)}_n\|}{\|v^{s^{-1}(j')}_n\|} \in \left[R - \log L_K,\; R + \log L_K\right]\right)\nonumber\\
\leq\;& C'\left(\frac{2\log L_K}{w^{j,j'}\sqrt{n}} + \frac{\log n}{\sqrt{n}}\right), \nonumber
\end{align}
where $\log L_K > 0$ is a uniform bound for $|\log(a^j/a^{j'})|$ as $j,j'\in J^i$ and $g\in K$ vary, and the last line uses the Berry--Esseen bound from the proof of Proposition~\ref{prop:accum}. Applying the telescoping argument to pass from pairwise bounds to the joint event: for events $A_k$, $B_k$,
\begin{align*}
\left|\mathbb{P}(\cap_k A_k) - \mathbb{P}(\cap_k B_k)\right|
&\leq \sum_l \biggl| \mathbb{P}\!\left((\cap_{k<l}A_k)\cap A_l\cap(\cap_{k>l}B_k)\right) \\
&\qquad\qquad - \mathbb{P}\!\left((\cap_{k<l}A_k)\cap B_l\cap(\cap_{k>l}B_k)\right) \biggr| \\
&\leq \sum_l \mathbb{P}(A_l\,\Delta\,B_l),
\end{align*}
we obtain, uniformly in $s$:
\begin{align}
&\left|\mathbb{P}\!\left(\log\frac{\|v^{s^{-1}(j)}_n\|}{\|v^{s^{-1}(j')}_n\|}+\log\frac{a^{s^{-1}(j)}_{n+1}}{a^{s^{-1}(j')}_{n+1}}>R,\;\forall j'\;\Big|\;\pi_{n+1}=s\right)\right.\nonumber\\
&\quad - \left.\mathbb{P}\!\left(\log\frac{\|v^{s^{-1}(j)}_n\|}{\|v^{s^{-1}(j')}_n\|}>R,\;\forall j'\right)\right| \nonumber\\
\leq\;& \sum_{j'} C'\left(\frac{R}{w^{j,j'}\sqrt{n}}+\frac{\log n}{\sqrt{n}}\right) \leq \frac{C}{\sqrt{n}}(R+\log n). \label{eqn:2}
\end{align}
Noting that $\mathbb{P}\!\left(\log\frac{\|v^{s^{-1}(j)}_n\|}{\|v^{s^{-1}(j')}_n\|} > R,\;\forall j'\right) = \mu^{*n}*\delta_x(U^{s^{-1}(j)}_R)$, we substitute \eqref{eqn:2} into \eqref{eqn:1}:
$$
\mu^{*(n+1)}*\delta_x(U^j_R) \leq \sum_{s\in F_\mu} \mu^{*n}*\delta_x(U^{s^{-1}(j)}_R)\;\mathbb{P}(\pi_1 = s) + \frac{C}{\sqrt{n}}(R+\log n).
$$
Setting $Q^{j,j_0} = \mathbb{P}(\pi_1(j_0)=j)$, this reads $\omega_{n+1} \leq Q\omega_n + C\frac{R+\log n}{\sqrt{n}}$ entry-wise. An induction yields the inequality of the lemma.
\end{proof}

Equipped with Lemma~\ref{lemma:recursion}, we complete the proof. Taking $R = O(\log n)$, there exists $C_1 > 0$ such that
$$
\omega _{n} \leq Q^{n}\omega _{1} + C_1 \sum_{m=1}^{n-1}Q^{n-m-1} \frac{1}{m^{1/3}}.
$$
Since $Q^{n}$ converges to the projection onto the uniform vector at exponential rate $\delta \in (0,1)$, we decompose $\omega_n = \omega_n^{(1)} + \omega_n^{(2)}$ into its projection onto the uniform vector and its orthogonal complement:
$$
\|\omega _{n}^{(2)}\| \leq C_1 \left( \delta ^{n} + C_2 \sum_{m=1}^{n-1} \frac{\delta ^{(n-m-1)/2}}{n^{1/3}} \right) \leq \frac{C_3}{n^{1/3}} .
$$
This gives $\|\omega _{n} - \omega _{n}^{(1)}\| \leq \frac{C_3}{n^{1/3}}$. Since all entries of $\omega_n^{(1)}$ are equal, we conclude $\omega ^{j,j'}_{n} \leq \frac{2 C_3}{n^{1/3}}$, which satisfies the first claim since $1/3$ is a universal bound $< 1/2$.

Integrating over the initial condition, write $\nu _{n}=\int_X \mu ^{*n}*\delta _{x} \, d\nu'(x)$. Using $B^{j} \setminus U^{j}_{R} \subset \bigcup_{j' \in J\setminus \left\{ j \right\}} U^{j',j}_{R}$ together with the accumulation bound of Proposition~\ref{prop:accum}, we obtain:
$$
\left| \nu _{n}(B^{j}) - \int_X  \mu ^{*n}*\delta _{x}(U^{j}_{R})\, d\nu'(x)  \right| \leq \frac{C_4}{n^{\gamma_4}}
$$
Thus:
$$
\left| \nu _{n}(B^{j'}) - \nu _{n}(B^{j}) \right| \leq  \frac{C_4}{n^{\gamma_4}} + \int_X \omega _{n}^{j,j'}  \, d\nu'(x) \leq \frac{C}{n^{\gamma}} .
$$
\end{proof}

% ----------------------------------------------------------------------
\section{Proof of the main result}\label{sec:proof}

We now have all the ingredients in place. The accumulation estimate (Proposition~\ref{prop:accum}) and the equidistribution estimate (Proposition~\ref{prop:equid}) give polynomial convergence of $\int_X \phi_\mu\,d\nu_n$ to $\lambda_1(\mu)$, while Lemma~\ref{lemma:estimate1} shows the difference $|\lambda_1(\mu') - \lambda_1(\mu)|$ is controlled by a combination of an exponential term in $d_\mathrm{W}(\mu',\mu)$ and this convergence rate. Balancing these two terms produces the log-H\"older bound.

\subsection{Convergence rate of the drift}

\begin{prop}[Convergence Rate] \label{prop:conv}
Let $K \subset G$ be a compact subset, $\mu$ and $\mu'$ be probability measures on $G$ supported in $K$. Assume $\mu$ is semisimple with $\lambda_{1}(\mu)=\lambda _{d}(\mu)$ and $\mu$ is aperiodic in $F_\mu$. Let $\nu'$ be a maximal $\mu'$-stationary probability measure and $\nu _{n} = P^{n}_{\mu}\nu'$. There exist a universal constant $\gamma > 0$ and a constant $C(K,\mu)>0$, such that for any $n \in \mathbb{N}^{*}$,
$$
\left| \int_X \phi _{\mu} \, d\nu_n - \lambda_{1}(\mu) \right| \leq \frac{C}{n^{\gamma}}.
$$
\end{prop}

\begin{proof}
Apply Lemma \ref{lemma:estimate2} to the probability measure $\xi = \nu_n$. For any $r \in \left(0, \frac{1}{2}\right)$, we have the bound:
$$
\left| \int_X \phi_{\mu} \, d\nu_n - \lambda_1(\mu) \right| \leq \sqrt{2} L_K \left( r + \nu_n(B_r^c) + \sum_{i \in I} \sum_{j,j' \in J^i} \left| \nu_n(B^j) - \nu_n(B^{j'}) \right| \right).
$$

We must choose $r$ depending on $n$ to optimize this bound. Let $r = n^{-\beta}$ for any universal $\beta > 1/2$. Recall from Lemma \ref{lemma:accum} that the parameter $R$ is defined as $R = -\log \frac{r}{\sqrt{d}}$. With our choice of $r$, we have:
$$
R = \beta \log n + \frac{1}{2} \log d.
$$
This ensures that $R = O(\log n)$ as $n \to \infty$.

By Proposition \ref{prop:accum}, there exist a constant $C_1 > 0$ and a universal exponent $\gamma_1 > 0$ such that:
$$
\nu_n(B_r^c) \leq \frac{C_1}{n^{\gamma_1}} (1 + \beta \log n).
$$
For any universally positive $\gamma_1' < \gamma_1$, the logarithmic factor is absorbed by the polynomial decay, yielding $\nu_n(B_r^c) \leq \frac{C_1'}{n^{\gamma_1'}}$. Because $R = O(\log n)$, Proposition \ref{prop:equid} is satisfied. Thus, there exist a constant $C_2 > 0$ and a universal $\gamma_2 > 0$ such that the permutation sum is bounded by:
$$
\sum_{i \in I} \sum_{j,j' \in J^i} \left| \nu_n(B^j) - \nu_n(B^{j'}) \right| \leq \frac{C_2}{n^{\gamma_2}}.
$$

Substituting these three decay rates back into the inequality, we obtain:
$$
\left| \int_X \phi_{\mu} \, d\nu_n - \lambda_1(\mu) \right| \leq\sqrt{2} L_K \left( \frac{1}{n^\beta} + \frac{C_1'}{n^{\gamma_1'}} + \frac{C_2}{n^{\gamma_2}} \right).
$$
Setting the universal constant $\gamma = \min(\beta, \gamma_1', \gamma_2) > 0$ and taking $C = \sqrt{2} L_K(1 + C_1' + C_2)$, we conclude that $\left| \int_X \phi_{\mu} \, d\nu_n - \lambda_1(\mu) \right| \leq \frac{C}{n^\gamma}$ for all $n \in \mathbb{N}^*$.
\end{proof}

\subsection{Concluding the theorem}

Proposition~\ref{prop:conv} gives polynomial control on the drift error when $\mu$ is aperiodic. We now handle the general case by reducing it to the aperiodic one via a convolution power, and then extend from $\lambda_1$ to the full spectrum via exterior powers.

\begin{proof} [Proof of Theorem \ref{thmA}]
Assume first that $\mu$ is aperiodic in $F_\mu$. Lemma \ref{lemma:estimate1} and Proposition \ref{prop:conv} give us the following estimate: for a universal constant $\gamma>0$ and $C_0(K,\mu)>1$, and any $n \in \mathbb{N}^{*}$,
$$
\left| \lambda_{1}(\mu') - \lambda_{1}(\mu) \right| \leq  \frac{C_0}{n^\gamma} + C_K^{n} d_{\mathrm W}(\mu',\mu) .
$$
Let $n \ge 0$ be the smallest such that 
$$
C_K^{n} d_{\mathrm W}(\mu',\mu) > 2^{-n}.
$$
We have always $n \geq 1$ and 
the reversed inequality: $C_K^{n-1} d_{\mathrm W}(\mu', \mu) \leq 2^{-(n-1)}$. We bound the difference, since $n^{\gamma} \leq 2^{n-1}$:
$$
\left| \lambda_{1}(\mu') - \lambda_{1}(\mu) \right| \leq   \frac{C_0}{n^{\gamma}} + \frac{C_K}{2^{n-1}}  \leq \frac{C_0+C_K}{n^{\gamma}} \leq C (-\log d_{\mathrm W}(\mu',\mu))^{-\gamma}.
$$
This concludes the result for $\lambda_1(\cdot)$ when $\mu$ is aperiodic.

Now assume $\mu$ is not aperiodic. The convolution $\mu ^{*p_{\mu}}$ is aperiodic in $A_{\mu}$ by Lemma~\ref{lemma:aper}, and 
$$
\lambda_{1}(\mu ^{*p_{\mu}}) = p_{\mu} \lambda_{1}(\mu) ,
$$
by Lemma~\ref{lemma:convolution_le}. Combining this with the Wasserstein bound of Lemma~\ref{lemma:trivial}, we have $d_{\mathrm W}( (\mu')^{*p_{\mu}} , \mu ^{*p_\mu} ) \leq p_{\mu} L_{K}^{(p_{\mu}-1)/2} d_{\mathrm W} ( \mu', \mu )$. Hence, uniformly whenever $p_{\mu} L_{K}^{(p_{\mu}-1)/2} d_{\mathrm W} ( \mu', \mu ) \leq 1$, 
for instance whenever $d_{\mathrm W} ( \mu', \mu ) \leq \left( 2 L_K \right)^{-{d!}}$:
\begin{align*}
\left| \lambda_{1}(\mu') - \lambda_{1}(\mu) \right|
&= \frac{1}{p_{\mu}} \left| \lambda_{1}((\mu')^{*p_{\mu}}) - \lambda_{1}(\mu^{*p_{\mu}}) \right| \\
&\leq \frac{C_0}{p_{\mu}} \left| \log d_{\mathrm W}\!\left((\mu')^{*p_{\mu}}, \mu^{*p_\mu}\right) \right|^{-\gamma}
\leq C \left| \log d_{\mathrm W}(\mu', \mu) \right|^{-\gamma}.
\end{align*}

For $\lambda_k$ with $k \geq 2$, we apply the exterior power reduction. 
It is clear that $\wedge^k\mu$ is also compactly supported, semisimple with one-point Lyapunov spectrum.
Since the log-H\"{o}lder bound for $\lambda_1$ is now established, the identity
$$
\lambda_k(\mu) = \lambda_1(\wedge^k \mu) - \lambda_1(\wedge^{k-1} \mu)
$$
yields log-H\"{o}lder continuity for all Lyapunov exponents.
\end{proof}

\bibliographystyle{alpha}
\bibliography{bib}

\end{document}